\documentclass{amsart}
\usepackage{amssymb}
\usepackage{mathrsfs}
\usepackage{amsmath}
\usepackage{amsfonts}
\usepackage{pifont}
\usepackage{graphicx,amssymb,mathrsfs,amsmath}
\usepackage{tocvsec2}
\usepackage{color}
\newtheorem{thm}{Theorem}[section]
\newtheorem{cor}[thm]{Corollary}
\newtheorem{lem}[thm]{Lemma}
\newtheorem{prop}[thm]{Proposition}
\newtheorem{prob}[thm]{Problem}
\theoremstyle{definition}
\newtheorem{defn}[thm]{Definition}
\newtheorem{example}[thm]{Example}
\theoremstyle{remark}
\newtheorem{rem}[thm]{Remark}
\numberwithin{equation}{section}
\begin{document}
\title[Reiterative $m_{n}$-distributional chaos of type $s$ in Fr\' echet spaces]{Reiterative $m_{n}$-distributional chaos of type $s$ in Fr\' echet spaces}

\author{Marko Kosti\' c}
\address{Faculty of Technical Sciences,
University of Novi Sad,
Trg D. Obradovi\' ca 6, 21125 Novi Sad, Serbia}
\email{marco.s@verat.net}

{\renewcommand{\thefootnote}{} \footnote{
\\ \text{  }  \ \    {\it Key words and phrases.} $m_{n}$-distributional chaos of type $s$, reiterative $m_{n}$-distributional chaos of type $s$, $\lambda$-distributional chaos of type $s$, reiterative $\lambda$-distributional chaos of type $s$, Fr\' echet spaces.
\\  \text{  }  \ \ The author is partially supported by grant 174024 of Ministry of Science and Technological Development, Republic of Serbia.}}

\begin{abstract}
The main aim of this paper is to consider various notions of (dense) $m_{n}$-distributional chaos of type $s$ and (dense) reiterative $m_{n}$-distributional chaos of type $s$ for general sequences of linear not necessarily continuous operators in Fr\' echet spaces. Here, $(m_{n})$ is an increasing sequence in $[1,\infty)$ satisfying $\liminf_{n\rightarrow \infty}\frac{m_{n}}{n}>0$ and $s$ could be $0,1,2,2+,2\frac{1}{2},3,1+,2-,2_{Bd},2_{Bd}+.$ We investigate $m_{n}$-distributionally chaotic properties and reiteratively $m_{n}$-distributionally chaotic properties of some special classes of operators like weighted forward shift operators and weighted backward shift operators in Fr\' echet sequence spaces, considering also continuous analogues of introduced notions and some applications to abstract partial differential equations.
\end{abstract}
\maketitle

\section{Introduction and Preliminaries}\label{intro}

Assume that $X$ is a Fr\' echet space. A
linear operator $T$ on $X$ is said to be hypercyclic iff there
exists an element $x\in D_{\infty}(T)\equiv \bigcap_{n\in {\mathbb
N}}D(T^{n})$ whose orbit $Orb(x,T)\equiv \{ T^{n}x : n\in {{\mathbb N}}_{0} \}$ is
dense in $X;$ $T $ is said to be topologically transitive,
resp. topologically mixing, iff for every pair of open
non-empty subsets $U,\ V$ of $X,$ there exists $n_{0}\in {\mathbb
N}$ such that $T^{n_{0}}(U) \ \cap \ V \neq \emptyset ,$ resp. there
exists $n_{0}\in {\mathbb N}$ such that, for every $n\in {\mathbb
N}$ with $n\geq n_{0},$ $T^{n}(U) \ \cap \ V \neq \emptyset .$ A
linear operator $T$ on $X$ is said to be chaotic iff it is topologically transitive and the set of periodic points of $T,$ defined by $\{x\in D_{\infty}(T) : (\exists n\in {\mathbb N})\, T^{n}x=x\},$ is dense in $X.$
For more details about topological dynamics of linear operators in Fr\' echet spaces, we refer the reader to the monographs \cite{bayart} by F. Bayart, E. Matheron and
\cite{erdper} by K.-G. Grosse-Erdmann, A. Peris.  

A strong motivational factor for genesis of this paper presents the fact that the structural results established in the foundational paper \cite{2013JFA} by N. C. Bernardes Jr. et al  
have not yet been seriously elucidated and completely reexamined for sequences of linear continuous operators in Fr\' echet spaces (cf. also the article \cite{mendoza} by J. A. Conejero et al). 
In our recent joint research study with A. Bonilla \cite{bk}, we have analyzed reiterative distributional chaos in Banach spaces and observed for the first time how the techniques developed in \cite{2013JFA} can be successfully applied in the analysis of dense Li-Yorke chaos. In this paper, which is
partly conceptualised as a certain addendum to the papers \cite{2013JFA} and \cite{bk}-\cite{mendoza}, the construction of distributionally irregular vectors developed in the proof of \cite[Theorem 15]{2013JFA} is essentially applied in the analysis of dense reiterative $m_{n}$-distributional chaos of type $s$ in Fr\' echet spaces. Besides that,
the paper contains a great number of other novelties, particularly those concerned with the deeper analysis of distributionally chaotic linear continuous operators $T$ on Banach space $X$ for which the condition   
\begin{align}\label{doktorbalija}
\lim_{N\rightarrow \infty}\frac{1}{N}\sum_{j=1}^{N}\bigl \|T^{j}x\bigr\|=+\infty\ \ \mbox{ for all } \ \ x\in X\setminus \{0\}
\end{align}
holds true (the existence of such an operator, acting on the space $X:=l^{p}({\mathbb N})$ for some $p\in [1,\infty)$ or $X:=c_{0}({\mathbb N}),$ has been recently proved in \cite[Theorem 25]{2018JMMA}). We revisit the Godefroy-Schapiro criterion and its continuous analogue, the Desch-Schappacher-Webb criterion, in our new framework. 

Concerning motivation, mention should be also made of the recent research study \cite{xiong} by J. C. Xiong, H. M. Fu and H. Y. Wang, where $\lambda$-distributionally chaotic continuous mappings between 
compact
metric spaces have been analyzed ($\lambda \in (0,1]$). Following this approach, we further specify linear distributional chaos in Fr\'echet spaces by using the concepts of lower and upper (Banach) $m_{n}$-densities and, in particular, the lower and upper (Banach) $\lambda$-densities.
Speaking-matter-of-factly, we analyze
various notions of (dense) $m_{n}$-distributional chaos of type $s$ and (dense) reiterative $m_{n}$-distributional chaos of type $s$ for general sequences of linear operators in Fr\' echet spaces, where $(m_{n})$ is an increasing sequence in $[1,\infty)$ satisfying $\liminf_{n\rightarrow \infty}\frac{m_{n}}{n}>0$ and $s$ symbolically takes one of values $0,1,2,2+,2\frac{1}{2},3,1+,2-,2_{Bd},2_{Bd}+$; before we go any further, we want to say that any $m_{n}$-distributionally chaotic sequence and, in particular, any $\lambda$-distributionally chaotic sequence, needs to be distributionally chaotic, i.e., $1$-distributionally chaotic, as well as that the converse statement is not true in general. %We investigate some special classes of operators and orbits of single linear 

The organization and main ideas of this paper can be briefly summarized as follows.
After collecting some necessary preliminaries about Fr\' echet spaces we are working with (separability or infinite-dimensionality is not assumed a priori), 
we provide basic definitions and properties of lower and upper (Banach) $m_{n}$-densities in Subsection \ref{publi}. In the second section of paper, we fix notions and introduce several different types of (reiterative) distributional chaos. In particular, we analyze distributional chaos of type
$s \in \{1,2,2_{Bd}+,2\frac{1}{2},3\},$ reiterative distributional chaos of type
$s \in \{0,1,1+,2-, 2,2+,2_{Bd},2_{Bd}+\}$ and Li-Yorke chaos; in Subsection \ref{pzi-dole}, we provide several observations, results and open problems for orbits of single operators in Fr\' echet spaces. In this subsection, we prove several results regarding the existence of distributionally chaotic operators which are not $m_{n}$-distributionally
chaotic for certain types of sequences $(m_{n}).$ In particular, the following is shown: Suppose that $X:=c_{0}({\mathbb N})$ or $X:=l^{p}({\mathbb N})$ for some $p\in [1,\infty).$
\begin{itemize} 
\item[(a)]  There exists a weighted forward shift operator $T\in L(X)$ which is $\lambda$-distributionally chaotic for any number $\lambda \in (0,1]$ and which additionally satisfies \eqref{doktorbalija}.
\item[(b)] For each number $\lambda \in (0,1],$ there exists a weighted forward shift operator $T\in L(X)$ satisfying \eqref{doktorbalija},
which is
$\lambda$-distributionally chaotic and not $\lambda'$-distributionally chaotic for any $\lambda'\in (0,\lambda).$ 
\item[(c)] For each numbers $a>0$ and $b \in (0,1),$ there exists a weighted forward shift operator $T\in L(X)$ satisfying \eqref{doktorbalija},
which is $(2^{an^{b}})$-distributionally chaotic.
\end{itemize}
Moreover, in the statements (a)-(c), the corresponding weight $(\omega_{j})_{j\in {\mathbb N}}$ can be chosen to consist of sufficiently large blocks of $2$'s and sufficiently large blocks of $(1/2)$'s.
In Proposition \ref{ekvivalentno}, we reconsider the notion of (dense) Li-Yorke chaos following an idea from \cite{bk}.
The main purpose of Section \ref{totijemarko} is to investigate associated notions of 
reiterative $m_{n}$-distributionally irregular vectors of type $s$ and reiterative $m_{n}$-distributionally irregular manifolds of type $s;$ in a series of results presented in Subsection \ref{profice}, we reconsider and slightly improve \cite[Proposition 7-Proposition 9, Theorem 12]{2013JFA} and  
\cite[Theorem 3.7, Corollary 3.12]{mendoza} for $m_{n}$-distributional chaos and $\lambda$-distributional chaos.
In contrast with orbits of linear continuous operators, the situation is much more complicated with general sequences because 
(reiterative) distributional chaos of type 
$s$ and Li-Yorke chaos can occur even in finite-dimensional spaces.
The main aim of Section \ref{marek-MLOss} is to analyze dense $m_{n}$-distributional chaos of type $s,$ dense reiterative $m_{n}$-distributional chaos of type $s$ 
and dense Li-Yorke chaos
for general sequences of linear operators in Fr\' echet spaces; as in our recent research study \cite{mendoza}, this 
section aims to show that the results about various types of dense (reiterative) distributional chaos of type
$s$ and dense Li-Yorke chaos can be formulated in this general setting. We extend the well-known result \cite[Theorem 3.5]{turkish-notes} by L. Luo and B. Hou, examining dense $m_{n}$-distributionally chaotic properties of the unilateral weighted backward shift operator in $l^{p}({\mathbb N})$ and $c_{0}({\mathbb N})$ spaces, where $1<p<\infty$ and the corresponding weight is given by $\omega_{j}:=2j/2j-1$ for $j\in {\mathbb N}$ (in contrast to the case $p=1$ considered in \cite{turkish-notes}, we obtain completely different results in case $p>1$).
In Subsection \ref{PDEs}, we introduce continuous analogues of the lower $m_{n}$-densities and, after clarifying some results for the families of linear not necessarily continuous operators defined on the non-negative real axis, we briefly explain how we can provide certain applications in the qualitative analysis of 
$f$-distributionally chaotic and $\lambda$-distributionally chaotic solutions of the abstract (fractional) partial differential equations in Fr\' echet spaces; here, $f : [0,\infty) \rightarrow [1,\infty)$ is an increasing mapping satisfying $\liminf_{t\rightarrow +\infty}\frac{f(t)}{t}>0.$
Section \ref{problemen} is reserved for giving final observations and open problems that we have not been able to solve (see also Problem \ref{prob1} and Problem \ref{prob2} proposed earlier, in Subsection \ref{pzi-dole}; the analysis of reiterative $m_{n}$-distributional chaos of type $s$ for composition operators is not carried out here).
Before explaining the notation used in the paper, the author would like to express his sincere gratitude to Prof. A. Bonilla, J. A. Conejero, M. Murillo-Arcila and X. Wu for many stimulating discussions during the preparation of manuscript.

Henceforth we assume that $X$ is a
Fr\' echet space over the field ${\mathbb K}\in \{{\mathbb R},\, {\mathbb C}\}$ and that the topology of $X$ is
induced by the fundamental system $(p_{n})_{n\in {\mathbb N}}$ of
increasing seminorms. The translation invariant metric\index{ translation invariant metric} $d :
X\times X \rightarrow [0,\infty),$ defined by
\begin{equation}\label{metri}
d(x,y):=\sum
\limits_{n=1}^{\infty}\frac{1}{2^{n}}\frac{p_{n}(x-y)}{1+p_{n}(x-y)},\
x,\ y\in X,
\end{equation}
enjoys the following properties:
$
d(x+u,y+v)\leq d(x,y)+d(u,v),\ x,\ y,\ u,\
v\in X,
$
$
d(cx,cy)\leq (|c|+1)d(x,y),\ c\in {\mathbb K},\ x,\ y\in X,
$
and
$
d(\alpha x,\beta x)\geq \frac{|\alpha-\beta|}{1+|\alpha-\beta|}d(0,x),\ x\in X,\ \alpha,\ \beta \in
{\mathbb K}.
$
By $Y$ we denote possibly another
Fr\' echet space over the same field of scalars as $X;$ the topology of $Y$ is induced by the
fundamental system $(p_{n}^{Y})_{n\in {\mathbb N}}$ of increasing
seminorms.
Define the translation invariant metric $d_{Y} :
Y\times Y \rightarrow [0,\infty)$ by replacing $p_{n}(\cdot)$ with
$p_{n}^{Y}(\cdot)$ in (\ref{metri}).
If
$(X,\|\cdot \|)$ or $(Y,\|\cdot \|_{Y})$ is a Banach space, then we assume that the
distance of two elements $x,\ y\in X$ ($x,\ y\in Y$) is given by $d(x,y):=\|x-y\|$ ($d_{Y}(x,y):=\|x-y\|_{Y}$).
Keeping in mind this terminological change,
our structural results clarified in Fr\' echet spaces continue to hold in the case that $X$ or $Y$ is a Banach space. 
If $A\subseteq {\mathbb N},$ we denote its complement by $A^{c};$ for any $s\in {\mathbb R},$ we set $\lfloor s \rfloor :=\sup \{
l\in {\mathbb Z} : s\geq l \}$ and $\lceil s \rceil:=\inf \{ l\in
{\mathbb Z} : s\leq l \}.$ Let us recall that an infinite subset $A$ of ${\mathbb N}$ is called syndetic iff its difference set, defined as usually, is bounded from above.

For a linear operator $A$ on $X,$ by $D(A),$ $R(A)$ and $\sigma_{p}(A)$ we denote its domain, range and point spectrum, respectively.
Suppose that $C\in L(X)$ is injective. Set $p_{n}^{C}(x):=p_{n}(C^{-1}x),$ $n\in {\mathbb N},$ $x\in R(C).$ Then
$p_{n}^{C}(\cdot)$ is a seminorm on $R(C)$ and the calibration
$(p_{n}^{C})_{n\in {\mathbb N}}$ induces a Fr\' echet locally convex topology on
$R(C);$ we denote this space simply by $[R(C)].$ Notice
that $[R(C)]$ is a Banach space provided that $X$ is. By $I$ we denote the identity operator on $X$ (in this paper, we analyze only single-valued linear operators; for various extensions in multi-valued setting and related results obtained recently, see \cite{ddc-emen-enes}-\cite{skopje} and our forthcoming monograph \cite{novascience}).

\subsection{Lower and upper densities}\label{publi}

Suppose that $A\subseteq {\mathbb N}.$ As it is well known, the lower density of $A,$ denoted by $\underline{d}(A),$ is defined by
$$
\underline{d}(A):=\liminf_{n\rightarrow \infty}\frac{|A \cap [1,n]|}{n},
$$
and the upper density of $A,$ denoted by $\overline{d}(A),$ is defined by
$$
\overline{d}(A):=\limsup_{n\rightarrow \infty}\frac{|A \cap [1,n]|}{n}.
$$
Further on, the lower Banach density of $A,$ denoted by $\underline{Bd}(A),$ is defined by
$$
\underline{Bd}(A):=\lim_{s\rightarrow +\infty}\liminf_{n\rightarrow \infty}\frac{|A \cap [n+1,n+s]|}{s}
$$
and the (upper) Banach density of $A,$ denoted by $\overline{Bd}(A),$ is defined by
$$
\overline{Bd}(A):=\lim_{s\rightarrow +\infty}\limsup_{n\rightarrow \infty}\frac{|A \cap [n+1,n+s]|}{s}.
$$
Then
\begin{align*}
0\leq \underline{Bd}(A) \leq \underline{d}(A) \leq \overline{d}(A) \leq \overline{Bd}(A)\leq 1,
\end{align*}
\begin{align*}
\underline{d}(A) + \overline{d}(A^{c})=1\mbox{ and } 
\underline{Bd}(A) +\overline{Bd}(A^{c})=1.
\end{align*}

The following notions of lower and upper densities for a subset $A\subseteq {\mathbb N}$ have been recently analyzed in \cite{F-operatori}:

\begin{defn}\label{prckojed} 
Suppose that $(m_{n})$ is an increasing sequence in $[1,\infty)$ and  $q\in [1,\infty).$ Then:
\begin{itemize}
\item[(i)] The lower $(m_{n})$-density of $A,$ denoted by $\underline{d}_{m_{n}}(A),$ is defined by
$$
\underline{d}_{{m_{n}}} (A):=\liminf_{n\rightarrow \infty}\frac{|A \cap [1,m_{n}]|}{n}.
$$
\item[(ii)] The upper $(m_{n})$-density of $A,$ denoted by $\overline{d}_{{m_{n}}}(A),$ is defined by
$$
\overline{d}_{{m_{n}}}(A):=\limsup_{n\rightarrow \infty}\frac{|A \cap [1,m_{n}]|}{n}.
$$
\item[(iii)] The lower $q$-density of $A,$ denoted by $\underline{d}_{q}(A),$ is defined by
$$
\underline{d}_{q}(A):=\liminf_{n\rightarrow \infty}\frac{|A \cap [1,n^{q}]|}{n}.
$$
\item[(iv)] The upper $q$-density of $A,$ denoted by $\overline{d}_{q}(A),$ is defined by
$$
\overline{d}_{q}(A):=\limsup_{n\rightarrow \infty}\frac{|A \cap [1,n^{q}]|}{n}.
$$

\end{itemize}
\end{defn}

We will use the following simple result:

\begin{lem}\label{perhane} 
Let  $q\geq 1$ and $A=\{ n_{1},\ n_{2}, \cdot \cdot \cdot,\  n_{k},\cdot \cdot \cdot \},$ where $(n_{k})$ is a strictly increasing sequence of positive integers. Then $\underline{d}_{q}(A)=\liminf_{k\rightarrow \infty}\frac{k}{n_{k}^{1/q}}$ and
$\underline{d}_{q}(A)>0$ iff there exists a finite constant $L>0$ such that $n_{k}\leq Lk^{q},$ $k\in {\mathbb N}.$ 
\end{lem}

In our further work, the following notion from \cite{F-operatori} will be crucially important:

\begin{defn}\label{guptinjo} 
Let $(m_{n})$ be an increasing sequence in $[1,\infty),$ $q\in [1,\infty)$ and $A\subseteq {\mathbb N}.$ Then we define:
\begin{itemize}
\item[(i)] The lower $l;(m_{n})$-Banach density of $A,$ denoted shortly by $\underline{Bd}_{l;{m_{n}}}(A),$
by
$$
\underline{Bd}_{l;{m_{n}}}(A):=\liminf_{s\rightarrow +\infty}\liminf_{n\rightarrow \infty}\frac{|A \cap [n+1,n+m_{s}]|}{s}.
$$ 
\item[(ii)] The lower $u;(m_{n})$-Banach density of $A,$ denoted shortly by $\underline{Bd}_{u;{m_{n}}}(A),$
by
$$
\underline{Bd}_{u;{m_{n}}}(A):=\limsup_{s\rightarrow +\infty}\liminf_{n\rightarrow \infty}\frac{|A \cap [n+1,n+m_{s}]|}{s}.
$$ 
\item[(iii)] The lower $l;q$-Banach density of $A,$ denoted shortly by $\underline{Bd}_{l;q}(A),$
by
$$
\underline{Bd}_{l;q}(A):=\liminf_{s\rightarrow +\infty}\liminf_{n\rightarrow \infty}\frac{|A \cap [n+1,n+s^{q}]|}{s}.
$$ 
\item[(iv)] The lower $u;q$-Banach density of $A,$ denoted shortly by $\underline{Bd}_{u;q}(A),$
by
$$
\underline{Bd}_{u;q}(A):=\limsup_{s\rightarrow +\infty}\liminf_{n\rightarrow \infty}\frac{|A \cap [n+1,n+s^{q}]|}{s}.
$$
\end{itemize}
\end{defn}

The above notion is not completely explored even in the case that $m_{n}=n^{q}$ for some $q>1.$ For example,
we have operated with $\liminf_{n\rightarrow \infty}$ in all above expressions but not with $\inf_{n\in {\mathbb N}},$ which will cause some troubles in the proof of Theorem \ref{na-dobrorc} below (these notions are no longer equivalent in general case $m_{n}\neq n$). In order the proof of Theorem \ref{na-dobro3} to work, we need to slightly modify the notion introduced in \cite{F-operatori} by operating with $\sup_{n\rightarrow \infty}$ in place of $\limsup_{n\rightarrow \infty}:$ 

\begin{defn}\label{prcko-prim-merak}
Let $(m_{n})$ be an increasing sequence in $[1,\infty),$ $q\in [1,\infty)$ and $A\subseteq {\mathbb N}.$ Then we define:
\begin{itemize}
\item[(i)] The (upper) $l:(m_{n})$-Banach density of $A,$ denoted shortly by $\overline{Bd}_{l:{m_{n}}}(A),$
by
$$
\overline{Bd}_{l:{m_{n}}}(A):=\liminf_{s\rightarrow +\infty}\sup_{n\in {\mathbb N}}\frac{|A \cap [n+1,n+m_{s}]|}{s}.
$$ 
\item[(ii)] The (upper) $u:(m_{n})$-Banach density of $A,$ denoted shortly by $\overline{Bd}_{u:{m_{n}}}(A),$
by
$$
\overline{Bd}_{u:{m_{n}}}(A):=\limsup_{s\rightarrow +\infty}\sup_{n\in {\mathbb N}}\frac{|A \cap [n+1,n+m_{s}]|}{s}.
$$ 
\item[(iii)] The $l:q$-Banach density of $A,$ denoted shortly by $\overline{Bd}_{l:q}(A),$
by
$$
\overline{Bd}_{l:q}(A):=\liminf_{s\rightarrow +\infty}\sup_{n\in {\mathbb N}}\frac{|A \cap [n+1,n+s^{q}]|}{s}.
$$ 
\item[(iv)] The $u:q$-Banach density of $A,$ denoted shortly by $\overline{Bd}_{u:q}(A),$
as follows
$$
\overline{Bd}_{u:q}(A):=\limsup_{s\rightarrow +\infty}\sup_{n\in {\mathbb N}}\frac{|A \cap [n+1,n+s^{q}]|}{s}.
$$ 
\end{itemize}
\end{defn}

The analysis of above densities is completely without scope of this paper and we only want to mention that $
\overline{Bd}_{l:q}(A)$ can be strictly greater than the quantity 
$$
\overline{Bd}_{l;q}(A):=\liminf_{s\rightarrow +\infty}\limsup_{n\rightarrow +\infty}\frac{|A \cap [n+1,n+s^{q}]|}{s}
$$ 
analyzed in \cite{F-operatori}, provided that $q>1.$ For example, if $A:=\{n^{2} : n\in {\mathbb N}\}$ and $q=2,$ then it can be easily seen that 
 $
\overline{Bd}_{l:q}(A)\geq 1$ as well that $\overline{Bd}_{l;q}(A)=0$ since for each $q>0$ and $s\in {\mathbb N}$ one has $\limsup_{n\rightarrow +\infty}\frac{|A \cap [n+1,n+s^{q}]|}{s}\leq 1/s.$ 

We will use the following simple lemma, as well
(cf. \cite{F-operatori} and \cite{skopje}):

\begin{lem}\label{mile-duo}
Suppose that $A\subseteq {\mathbb N}.$
\begin{itemize}
\item[(i)] Let $\liminf_{n\rightarrow \infty}\frac{m_{n}}{n}>0.$ Then $
\underline{Bd}_{l;m_{n}}(A)=0$ iff $
\underline{Bd}_{u;m_{n}}(A)=0$ iff $A$ is finite or $A$ is infinite non-syndetic.
\item[(ii)] Let $\liminf_{n\rightarrow \infty}\frac{m_{n}}{n}>0.$ Then $
\underline{d}_{m_{n}}(A)>0$ provided that $A$ is syndetic.
\end{itemize}
\end{lem} 

\section{Reiterative $m_{n}$-distributional chaos of type $s$}\label{marek-MLOs}

Denote by ${\mathrm R}$ the class consisting of all increasing sequences  $(m_{n})$ of positive reals satisfying $\liminf_{n\rightarrow \infty}\frac{m_{n}}{n}>0,$ i.e., there exists a finite constant $L>0$ such that $n\leq L m_{n},$ $n\in {\mathbb N}.$ If this is the case, then for each positive constant $a>0$ we have that 
$(am_{n})\in {\mathrm R}.$ Unless stated otherwise, we assume that $(m_{n})\in {\mathrm R}$
henceforth.

%The following slight extension of \cite[Lemma 2.1]{skopje} holds good:
%
%\begin{lem}\label{1krozel}
%Suppose that $A\subseteq {\mathbb N}$ and $(m_{n})\in {\mathrm R}.$ Then 
%$\underline{d}_{m_{n}}(A)+\overline{d}_{m_{n}}(A^{c})\geq \frac{1}{L}$ and $\underline{d}_{m_{n}}(A)\geq L^{-1}\underline{d}(A).$
%\end{lem}

In this section, it is assumed that, for every $j\in {\mathbb N},$ $T_{j} : D(T_{j}) \subseteq X \rightarrow Y$ is a linear operator, $T : D(T) \subseteq X \rightarrow X$ is a linear operator and $\tilde{X}$ is a non-empty
subset of $X.$
Let a number $\delta > 0$ and two elements $x,\ y\in \bigcap_{j\in {\mathbb N}}D(T_{j}) \bigcap \tilde{X}$ be given. We set
\begin{align*}
F_{x,y,m_{n}}(\delta):= \underline{d}_{m_{n}}\Bigl(\bigl\{j \in {\mathbb N} : d_{Y}\bigl(T_{j} x,T_{j} y\bigr) < \delta \bigr\}\Bigr),
\end{align*}
\begin{align*}
G_{x,y,m_{n}}(\delta):= \underline{d}_{m_{n}}\Bigl(\bigl\{j \in {\mathbb N} : d_{Y}\bigl(T_{j} x,T_{j} y\bigr) \geq \delta \bigr\}\Bigr),
\end{align*}
\begin{align*}
H_{x,y,m_{n}}(\delta):= \overline{d}_{m_{n}}\Bigl(\bigl\{j \in {\mathbb N} : d_{Y}\bigl(T_{j} x,T_{j} y\bigr) < \delta \bigr\}\Bigr),
\end{align*}
and
\begin{align*}
I_{x,y,m_{n}}(\delta):= \overline{d}_{m_{n}}\Bigl(\bigl\{j \in {\mathbb N} : d_{Y}\bigl(T_{j} x,T_{j} y\bigr) \geq \delta \bigr\}\Bigr).
\end{align*}

If a pair $(x,y)$ satisfies

 {(DC1)} $G_{x,y,m_{n}} \equiv 0$ and $F_{x,y,m_{n}}(\sigma) = 0$
              for some $\sigma > 0$, or

 {(DC2)} $G_{x,y,m_{n}} \equiv 0$ and $I_{x,y,m_{n}}(\sigma) >0$
              for some $\sigma > 0$, or

 {(DC2$\frac{1}{2}$)} There exist real numbers $c > 0$ and $r > 0$ such that
              $F_{x,y,m_{n}}(\delta) < c < H_{x,y,m_{n}}(\delta)$
              for all $ 0 < \delta < r,$ or

 {(DC3)} There exist real numbers $a > 0,$ $b > 0$ and $c>0$ such that
              $F_{x,y,m_{n}}(\delta) < c < H_{x,y,m_{n}}(\delta)$
              for all $ a < \delta < b,$

then $(x,y)$ is called a $\tilde{X}_{m_{n}}$-{\em distributionally chaotic pair of type}
$s \in \{1,2,2\frac{1}{2},3\}$ for $(T_{j})_{j\in
{\mathbb N}}$. The sequence $(T_{j})_{j\in
{\mathbb N}}$ is said to be
$\tilde{X}_{m_{n}}$-{\em distributionally chaotic of type} $s$ ($m_{n}$-{\em distributionally chaotic of type} $s,$ if $\tilde{X}=X$) if there exists an uncountable set
$S \subseteq X$ such that every pair $(x,y)$ of distinct points in
$S$ is a $\tilde{X}_{m_{n}}$-distributionally chaotic pair of type $s$ for $(T_{j})_{j\in
{\mathbb N}};$  the sequence $(T_{j})_{j\in
{\mathbb N}}$ is said to be
$\tilde{X}_{m_{n}}$-{\em distributionally chaotic} ($m_{n}$-{\em distributionally chaotic,} if $\tilde{X}=X$) iff there exist an uncountable set
$S \subseteq X$ and a number $\sigma>0$ such that for every pair $(x,y)$ of distinct points in
$S$ we have $G_{x,y,m_{n}} \equiv 0$ and $F_{x,y,m_{n}}(\sigma) = 0.$ Furthermore, if $S$ can be chosen to be dense in $\tilde{X},$ then we say that $(T_{j})_{j\in
{\mathbb N}}$ is {\em densely}
$\tilde{X}_{m_{n}}$-{\em distributionally chaotic of type} $s$ ({\em densely} $m_{n}$-{\em distributionally chaotic of type} $s,$ if $\tilde{X}=X$); {\em densely}
$\tilde{X}_{m_{n}}$-{\em distributionally chaotic} ({\em densely} $m_{n}$-{\em distributionally chaotic,} if $\tilde{X}=X$).

A linear operator $T : D(T) \subseteq X \rightarrow X$ is said to be ({\em densely})
$\tilde{X}_{m_{n}}${\em-distributionally chaotic} ({\it of type}) $s$ (({\em densely}) $m_{n}${\em-distributionally chaotic} ({\it of type} $s$), if $\tilde{X}=X$) iff the sequence $(T_{j}\equiv
T^{j})_{j\in {\mathbb N}}$ is. In this case,
$S$ is called a $\tilde{X}_{m_{n}}$-{\em distributionally scrambled set} ({\it of type $s$}) for the sequence $(T_{j})_{j\in
{\mathbb N}}$ (the operator $T$).

As mentioned on \cite[p. 798]{2018JMMA}, the notion of (dense) distributional chaos of type $1$ and the notion of (dense) distributional chaos coincide for operators on Fr\' echet spaces; after establishing Theorem \ref{ruza}, we will see that the same statement holds for the notion of (dense) $m_{n}$-distributional chaos of type $1$ and the notion of (dense) $m_{n}$-distributional chaos. On the other hand, the situation is completely different for general sequences of linear continuous operators:

\begin{example}\label{idioq}
It is clear that there exist two infinite sets $A,\ B\subseteq {\mathbb N}$ such that $\overline{d}(A)=\overline{d}(B)=1$ and ${\mathbb N}=A \cup B.$ Set $X:={\mathbb K},$
$T_{j}:=2I$ ($j\in A$) and $T_{j}:=0$ ($j\in B$). Using the fact that the set $S-S$ cannot be bounded away from zero if $S$ is an uncountable subset of ${\mathbb K},$ we can simply show that the sequence $(T_{j})$ is distributionally chaotic of type $1$ and not distributionally chaotic.
\end{example}

Before proceeding further, we would like to note that the $\lambda$-distributional chaos of sequence $ (T_{j})_{j\in {\mathbb N}}$ for some $\lambda \in (0,1]$ implies $\lambda'$-distributional chaos of $ (T_{j})_{j\in {\mathbb N}}$ for all numbers $\lambda'\in [\lambda,1].$ The converse statement is not true in general, as we will see later.

Next, for a given number $\delta > 0$, we set

\begin{align*}
BF_{x,y,m_{n}}(\delta):= \underline{Bd}_{l;m_{n}}\Bigl(\bigl\{j \in {\mathbb N} : d_{Y}\bigl(T_{j} x,T_{j} y\bigr) < \delta \bigr\}\Bigr),
\end{align*}
\begin{align*}
BG_{x,y,m_{n}}(\delta):= \underline{Bd}_{l;m_{n}}\Bigl(\bigl\{j \in {\mathbb N} : d_{Y}\bigl(T_{j} x,T_{j} y\bigr) \geq \delta \bigr\}\Bigr),
\end{align*}
and
\begin{align*}
BI_{x,y,m_{n}}(\delta):=  \overline{Bd}_{l:m_{n}}\Bigl(\bigl\{j \in {\mathbb N} : d_{Y}\bigl(T_{j} x,T_{j} y\bigr) \geq \delta \bigr\}\Bigr).
\end{align*}

\begin{defn}\label{uvodjenje}
We say that the sequence $(T_{j})_{j\in
{\mathbb N}}$ is $\tilde{X}_{m_{n}}$-{\it reiteratively distributionally chaotic of type $1 $} iff there exists an uncountable set $S\subseteq \bigcap_{j\in {\mathbb N}}D(T_{j}) \bigcap \tilde{X}$ such that
for every pair $(x,y) \in S \times S$ of distinct points $BF_{x,y,m_{n}}(\sigma) = 0$
for some $\sigma > 0$ and there exist $c \in (0,\liminf_{n\rightarrow \infty}\frac{m_{n}}{n})$ and $r > 0$ such that $G_{x,y,m_{n}}(\delta)\leq c$ for all $ 0 < \delta < r$. If $G_{x,y,m_{n}} \equiv 0$, we say that $(T_{j})_{j\in
{\mathbb N}}$ is $\tilde{X}_{m_{n}}$-{\it reiteratively distributionally chaotic of type $1+$.}

We say that $(T_{j})_{j\in
{\mathbb N}}$ is $\tilde{X}_{m_{n}}$-{\it reiteratively distributionally chaotic of type $2$} iff there exists an uncountable set $S\subseteq \bigcap_{j\in {\mathbb N}}D(T_{j}) \bigcap \tilde{X}$ such that
for every pair $(x,y) \in S \times S$ of distinct points we have $BG_{x,y,m_{n}}\equiv 0$  and $I_{x,y,m_{n}}(\sigma) >0$
for some $\sigma > 0$. 
Finally, we say that $(T_{j})_{j\in
{\mathbb N}}$ is  $\tilde{X}_{m_{n}}$-{\it (reiteratively) distributionally chaotic of type $2_{Bd}$} iff there exists an uncountable set $S\subseteq \bigcap_{j\in {\mathbb N}}D(T_{j}) \bigcap \tilde{X}$ such that
for every pair $(x,y) \in S \times S$ of distinct points we have ($BG_{x,y,m_{n}}\equiv 0$) $G_{x,y,m_{n}}\equiv 0$ and $BI_{x,y,m_{n}}(\sigma) >0$
for some $\sigma > 0$.

In the case that the number $\sigma>0$ does not depend on the choice of pair $(x,y) \in S \times S$, then we say that $(T_{j})_{j\in
{\mathbb N}}$ is $\tilde{X}_{m_{n}}$-{\it reiteratively distributionally chaotic of type $2+$} or $\tilde{X}_{m_{n}}$-{\it (reiteratively) distributionally chaotic of type $2_{Bd}+.$}
\end{defn}

A series of elaborate and very plain counterexamples shows that the conclusions established in our previous research study \cite{bk} are no longer true for general sequences of linear continuous operators, even on finite-dimensional spaces (for more details about this problematic, see \cite{ddc}). This also holds for $\tilde{X}_{m_{n}}$-reiterative distributional chaos of types $0,$ $1^{+}$ and $2-,$ which are introduced as follows: 

\begin{defn}\label{nizovi}
We say that the sequence $(T_{j})_{j\in
{\mathbb N}}$ is
$\tilde{X}_{m_{n}}$-{\it reiteratively distributionally chaotic of type $0,$} resp. $1^{+}$ [$2-$], iff 
there exist an uncountable
set $S \subseteq \bigcap_{j\in {\mathbb N}}D(T_{j}) \bigcap \tilde{X}$ and
$\sigma>0$ such that for each pair $x,\
y\in S$ of distinct points we have $BF_{x,y,m_{n}}(\sigma)=0$ and $BG_{x,y,m_{n}}\equiv 0, $ 
resp. $BF_{x,y,m_{n}}(\sigma)=0$ and $G_{x,y,m_{n}}\equiv 0$ [$F_{x,y,m_{n}}(\sigma)=0$ and $BG_{x,y,m_{n}}\equiv 0$]. 
\end{defn}

The notions of {\it reiteratively $\tilde{X}_{m_{n}}$-distributionally scrambled set of type $s$} and {\it dense } $\tilde{X}_{m_{n}}$-{\it reiterative distributional chaos of type $s,$} where \\ $s\in \{1,1^{+},1{+},2,2+,2_{Bd},2_{Bd}+,0,2-\}$ are introduced as above. 

It is worth noting that the notions introduced in Definition \ref{uvodjenje} and Definition \ref{nizovi} 
extend the notions introduced in \cite[Definition 1.3]{bk}, where we have only analyzed
the uniform reiterative distributional chaos of type $s$; e.g., in \cite{bk}, we have consider only the notion of reiterative distributional chaos of type $2+$ but not of type $2.$ 
Some implications trivially can be clarified, like reiterative distributional chaos of type $1^{+}$ implies reiterative distributional chaos of type $1+,$ which further implies reiterative distributional chaos of type $1.$

We will accept the following agreements. If $\tilde{X}=X$ or $m_{n}\equiv n,$ then we remove the prefixes ``$\tilde{X}$-'' and ``$m_{n}$-'' from the terms and notions. For example, dense reiterative
distributional chaos of type $3$ is dense reiterative $X_{n}$-distributional chaos of type $3.$ 
If $m_{n}\equiv n^{1/\lambda}$ for some $\lambda \in (0,1],$ then,
as a special case of the above definitions, we obtain the notions  of (dense, reiterative) $\tilde{X}_{\lambda}$-distributional chaos of type $s$,
(dense, reiterative) $\lambda$-distributional chaos of type $s$, (reiterative) $\tilde{X}_{\lambda}$-scrambled sets of type $s$ and (reiterative) $\lambda$-scrambled sets of type $s$ for operators and their sequences; 
for $\lambda=1,$ we obtain the notions of (dense, reiterative) $\tilde{X}$-distributional chaos of type $s$,
(dense, reiterative) distributional chaos of type $s$, (reiterative) $\tilde{X}$-scrambled sets of type $s$ and (reiterative) scrambled sets of type $s$ for operators and their sequences, and so on and so forth.

\begin{rem}\label{tres}
\begin{itemize}
\item[(i)] It is clear that, if $(m_{n}')\in {\mathrm R}$ and $m_{n}\leq m_{n}',$ $n\in {\mathbb N},$ then $(m_{n}')$-distributional chaos/$(m_{n}')$-distributional chaos of type $1$ or $2$ implies $(m_{n})$-distributional chaos/$(m_{n})$-distributional chaos of the same type, while reiterative $(m_{n}')$-distributional chaos of type $1,1^{+},1+,0$ or $2-$ implies reiterative $(m_{n})$-distributional chaos of the same type.
\item[(ii)] If $(m_{n})\in {\mathrm R},$ then $(m_{n})$-distributional chaos/$(m_{n})$-distributional chaos of type $1$ or $2$ implies distributional chaos/distributional chaos of the same type, while reiterative $(m_{n})$-distributional chaos of type $1,1^{+},1+,0$ or $2-$ implies reiterative distributional chaos of the same type. In particular, any $m_{n}$-distributionally chaotic sequence is distributionally chaotic. 
\item[(iii)]  It is worth noting that the distributional chaos of type $s\in \{1,2,3\}$ for backward shift operators in K\"othe sequence spaces has been analyzed for the first time by X. Wu et al \cite{xwuchen}, under certain assumptions other from ours.
\end{itemize}
\end{rem}

Finally, we will use the notion of Li-Yorke chaos below.
Li-Yorke chaos in Fr\'echet spaces has been recently investigated by N. C. Bernardes Jr et al \cite{band} and M. Kosti\' c \cite{nis-dragan} (see also \cite{2011}, \cite{ARXIV} and \cite{xwuchen}-\cite{xwu}): 

\begin{defn}\label{obori}
We say that the sequence
$(T_{j})_{j\in {\mathbb N}}$ is  $\tilde{X}$-{\em Li-Yorke chaotic}
iff there exists an uncountable set $S\subseteq \bigcap_{j\in {\mathbb N}}D(T_{j}) \bigcap \tilde{X}$ such that
for every pair $(x,y) \in S \times S$ of distinct points, we have
\[
 \liminf_{j \to \infty} d_{Y}\bigl(T_{j} x,T_{j} y\bigr) = 0
 \ \ \mbox{ and } \ \
 \limsup_{j \to \infty} d_{Y}\bigl(T_{j} x,T_{j} y\bigr)> 0.
\]
In this case, $S$ is called a  $\tilde{X}$-{\em scrambled set} for $(T_{j})_{j\in {\mathbb N}}$ and each
such pair $(x,y)$ is called a  $\tilde{X}$-{\em Li-Yorke pair} for $(T_{j})_{j\in {\mathbb N}}$. We say that $(T_{j})_{j\in {\mathbb N}}$ is {\it densely  $\tilde{X}$-Li-Yorke chaotic} iff $S$ can be chosen to be dense in $\tilde{X}.$
\end{defn}

We refer the reader to \cite{band} and \cite{nis-dragan} for the notion of Li-Yorke (semi-)irregular vectors.
Any notion introduced above is accepted also for 
a single linear continuous operator $T : D(T) \subseteq X \rightarrow X$ by using the sequence $(T_{j}\equiv
T^{j})_{j\in {\mathbb N}}$ for definition.  Finally, if $\tilde{X}=X,$ then we remove the prefix ``$\tilde{X}$-'' from the terms and notions.

Before we move ourselves to the next subsection, we  would like to present the following illustrative example:

\begin{example}\label{zelje}
Due to Proposition \ref{perhane},
if $\lambda \in (0,1]$ and $A=\{ n_{1},\ n_{2}, \cdot \cdot \cdot,\  n_{k},\cdot \cdot \cdot \},$ where $(n_{k})$ is a strictly increasing sequence of positive integers, then $\underline{d}_{1/\lambda}(A)=
0$ iff for any finite constant $L>0$ there exists $k\in {\mathbb N}$ such that $n_{k}> Lk^{1/\lambda}.$ Therefore, it is very simple to construct two disjoint subsets $A$ and $B$ of ${\mathbb N}$ such that ${\mathbb N}=A\cup B$ and $\underline{d}_{1/\lambda}(A)=\underline{d}_{1/\lambda}(B)=0$ for each number $\lambda \in (0,1];$ for example, set
$a_{n}:=\sum_{i=1}^{n}2^{2^{i^{2}}}$ ($n\in {\mathbb N}$), $A:=\bigcup_{n\in 2{\mathbb N}}[a_{n},a_{n+1}]$ and $B:={\mathbb N} \setminus A.$ After that, set 
$X:={\mathbb K}$, $T_{j}:=jI$ ($j\in A$) and $T_{j}:=0$ ($j\in B$). Then it can be simply checked that the sequence $(T_{j})_{j\in
{\mathbb N}}$ is densely
$\lambda$-distributionally chaotic for each number $\lambda \in (0,1]$, and that the corresponding scrambled set $S$ can be chosen to be the whole space $X.$ 
\end{example} 

\subsection{A few remarks and open problems for orbits of single operators}\label{pzi-dole}

We start this subsection with the observation that
the property (DC2) with $m_{n}\equiv n$ for $(T_{j})_{j\in
{\mathbb N}}$ is not equivalent with {\em the mean Li-Yorke chaos} for $(T_{j})_{j\in
{\mathbb N}},$ as for continuous mappings on compact metric spaces (cf. \cite{ARXIV} and \cite{2018JMMA}). Counterexamples exist even for orbits of weighted forward shift operators on $l^{p}$ spaces, where $1\leq p<\infty;$ for more details about the mean Li-Yorke chaos, the reader may consult \cite{ARXIV}, \cite{down} and references cited therein.

Further on, let us recall that there is no Li-Yorke chaotic operator on a Fr\' echet space that is compact (\cite{band}), so that the situation appearing in Example \ref{zelje} cannot occur for orbits on finite-dimensional spaces. Concerning infinite-dimensional spaces, the situation is completely different: there exists a continuous linear operator $T$ on $c_{0}({\mathbb N})$ or $ l^{p}({\mathbb N}),$ where $1\leq p<\infty,$ which is $\lambda$-distributionally chaotic for any number $\lambda \in (0,1],$ not hypercyclic and which additionally satisfies some other requirements. 
To verify this, we will prove the following proper extension of \cite[Theorem 25]{2018JMMA}; cf. also \cite[Remark 21]{2011}, \cite{ddc} and Theorem \ref{primer-primexsimex}:

\begin{thm}\label{primer}
Suppose that $X:=c_{0}({\mathbb N})$ or $X:=l^{p}({\mathbb N})$ for some $p\in [1,\infty).$ Then there exists a continuous linear operator $T$ on $X$ which is $\lambda$-distributionally chaotic for any number $\lambda \in (0,1]$ and which additionally satisfies \eqref{doktorbalija} 
as well as
$\lim_{j\rightarrow \infty}T^{j}x=0$ for some $x\in X$ iff $x=0.$ 
\end{thm}

\begin{proof}
Without loss of generality, we may assume that $X=l^{2}({\mathbb N})= l^{2}.$ Consider
a weighted forward shift $T\equiv F_{\omega}: 
l^{2} \rightarrow l^{2},$ defined by $F_{\omega} (x_{1}, x_{2}, \cdot \cdot \cdot) \mapsto (0, \omega_{1}x_{1},\omega_{2}x_{2},\cdot \cdot \cdot),$ where the sequence of weights $\omega =(\omega_{k})_{k\in {\mathbb N}}$ consists of sufficiently large blocks of $2$'s of lengths $b_{1}, \ b_{2},\cdot \cdot \cdot,$ and sufficiently large blocks of $(1/2)$'s of lengths $a_{1}, \ a_{2},\cdot \cdot \cdot .$ 
More precisely, let $b_{n}:=2^{2^{(2n-1)^{2}}}$ 
and $a_{n}:=2^{2^{(2n)^{2}}}$ ($n\in {\mathbb N}$). Let a number $\lambda \in (0,1)$  be fixed. To see that $T$ is $\lambda$-distributionally chaotic,
 it suffices to show that for each $n\in {\mathbb N}$ and we have $\underline{d}_{1/\lambda}(\{j\in {\mathbb N} : \|T^{j}e_{1}\|<2^{n}\})=0$
and $\underline{d}_{1/\lambda}(\{j\in {\mathbb N} : \|T^{j}e_{1}\|>2^{-n}\})=0,$ where $e_{1} = (1, 0, 0,\cdot \cdot \cdot).$ Towards this end,
observe that there exist a finite subset $A\subseteq {\mathbb N}$ and a number $n_{0}\in {\mathbb N}$ 
such that $\{j\in {\mathbb N} : \|T^{j}e_{1}\|\leq 2^{n}\} \subseteq A\cup \bigcup_{s\geq n_{0}}[2(b_{1}+b_{2}+\cdot \cdot \cdot +b_{s})-n,2(a_{1}+a_{2}+\cdot \cdot \cdot +a_{s})+n]\equiv B.$ Let $L>0$ be arbitrarily chosen. Then there exists sufficiently large $s\in {\mathbb N}$ such that, with $j=(2n+1)(s-1)+\sum_{w=1}^{s-1}(s-w)(a_{w}-b_{w})$ and $n_{j}=2\sum_{w=1}^{s}b_{w}-n,$ we have $n_{j}>Lj^{1/\lambda}.$ Therefore, $\underline{d}_{1/\lambda}(B)=0$ due to  Proposition \ref{perhane}. The second equation can be proved analogously, so that $T$ is $\lambda$-distributionally chaotic. Furthermore, it is clear that $T$ cannot be hypercyclic as well as that the condition $\lim_{j\rightarrow \infty}T^{j}x=0$ for some $x\in l^{2}$ implies $x=0.$
To complete the proof, it suffices to show that \eqref{doktorbalija} holds with $x=e_{1}.$ To estimate the term $\frac{1}{N}\sum_{j=1}^{N}\|T^{j}e_{1}\|,$ we will consider separately two possible cases:

1. There exists $n\in {\mathbb N}$ such that  $N\in [b_{n}+\sum_{s=1}^{n
-1}(a_{s}+b_{s}),\sum_{s=1}^{n
}(a_{s}+b_{s})].$ Then we have
\begin{align}
\notag \frac{1}{N}&\sum_{j=1}^{N}\bigl \|T^{j}e_{1}\bigr\|=\frac{1}{N}\sum_{j=1}^{N}\omega_{1}\omega_{2}\cdot \cdot \cdot \omega_{j}\geq \frac{1}{2na_{n}}\sum_{j=\sum_{s=1}^{n-1}2a_{s}}^{b_{n}+\sum_{s=1}^{n
-1}(a_{s}+b_{s})}\omega_{1}\omega_{2}\cdot \cdot \cdot \omega_{j}
 \\\notag & \geq   \frac{1}{2na_{n}}\sum_{j=0}^{b_{n}-a_{n-1}}2^{j}\geq \frac{2^{b_{n}-a_{n-1}}
-1}{2na_{n}}\geq \frac{2^{a_{n-1}}
-1}{2na_{n}}\rightarrow +\infty,\quad n\rightarrow +\infty.
\end{align}

2. There exists $n\in {\mathbb N}$ such that  $N\in [\sum_{s=1}^{n
-1}(a_{s}+b_{s}),b_{n}+\sum_{s=1}^{n
-1}(a_{s}+b_{s})].$ Then we have
\begin{align}
\notag \frac{1}{N}&\sum_{j=1}^{N}\bigl \|T^{j}e_{1}\bigr\|=\frac{1}{N}\sum_{j=1}^{N}\omega_{1}\omega_{2}\cdot \cdot \cdot \omega_{j}\geq \frac{1}{2nb_{n}}\sum_{j=\sum_{s=1}^{n-2}2a_{s}}^{b_{n-1}+\sum_{s=1}^{n
-2}(a_{s}+b_{s})}\omega_{1}\omega_{2}\cdot \cdot \cdot \omega_{j}
 \\\notag & \geq   \frac{1}{2nb_{n}}\sum_{j=0}^{b_{n-1}-a_{n-2}}2^{j}\geq \frac{2^{b_{n-1}-a_{n-2}}
-1}{2nb_{n}}\geq \frac{2^{a_{n-2}}
-1}{2nb_{n}}\rightarrow +\infty,\quad n\rightarrow +\infty.
\end{align}
The proof of the theorem is thereby complete.
\end{proof}

\begin{rem}\label{inverz-lambda}
Our construction is much more easier and transparent than the construction employed in \cite[Theorem 25]{2018JMMA} for case $\lambda=1.$ 
Arguing as in \cite[Remark 26]{2018JMMA} and the proof of Theorem \ref{primer}, we can prove the existence of an invertible continuous linear operator $T$ on $X:=c_{0}({\mathbb Z})$ or $X:=l^{p}({\mathbb Z})$ for some $p\in [1,\infty),$ satisfying the all required properties from Theorem \ref{primer}.
\end{rem}

Let $\lambda \in (0,1].$ 
Recall that the class of dynamical systems on the plane ${\mathbb R}^{2},$ named winding systems, have been introduced in \cite[Section 5]{xiong} in order to show that the
$\lambda$-distributional chaos of an operator $T$ on a compact metric space does not imply $\lambda'$-distributional chaos of $T$  for any number $\lambda '\in (0,\lambda).$ 
Before solving in the affirmative the corresponding problem for orbits of single-valued linear operators in Banach spaces satisfying the equation \eqref{doktorbalija}, 
we would like to show (cf. also Corollary \ref{cea1} and Example \ref{idiote} below) that there exist a continuous linear operator $T$ on $X:=c_{0}({\mathbb N})$ and a sequence $(P_{j}(z))_{j\in {\mathbb N}}$ of non-zero real polynomials such that sequence $ (T_{j}\equiv P_{j}(T))_{j\in {\mathbb N}}$ in $L(X)$ satisfies that there exists a dense linear submanifold $X_{0}$ of $X$ with $\lim_{j\rightarrow \infty}T_{j}x=0,$ $x\in X_{0},$
as well as that $ (T_{j})_{j\in {\mathbb N}}$ is 
$\lambda$-distributionally chaotic and not $\lambda'$-distributional chaotic for any $\lambda'\in (0,\lambda):$ 

\begin{example}\label{manjoza}
Suppose that $0<\lambda'<\lambda ,$ $X:=c_{0}({\mathbb N}),$ $a_{n}:=\lfloor n^{2/\lambda}\ln n\rfloor$ and $b_{n}:=\lfloor n^{2/\lambda}\ln n\rfloor +n$ ($n\in {\mathbb N}$). Set $S:=\bigcup_{n\geq 2}[b_{n-1},a_{n}).$  Then it can be easily seen that
$$
\underline{d}_{1/\lambda}\bigl(S^{c}\bigr)\leq \liminf_{n\rightarrow \infty} \frac{1+2+\cdot \cdot \cdot +n}{a_{n}^{\lambda}}=0
$$ and 
$$
\underline{d}_{1/\lambda'}\bigl(S^{c}\bigr)\geq \liminf_{n\rightarrow +\infty}\min\Biggl(\frac{1+2+\cdot \cdot \cdot +(n-1)}{b_{n}^{\lambda'}},\frac{1+2+\cdot \cdot \cdot +n}{a_{n}^{\lambda'}}\Biggr)=+\infty .
$$
Set $T\langle x_{1},x_{2},x_{3},\cdot \cdot \cdot \rangle:=\langle 2x_{2},2x_{3},2x_{4},\cdot \cdot \cdot \rangle$ 
for all $\langle x_{n}\rangle_{n\in {\mathbb N}},$ as well as $T_{j}:=T^{j},$ if $j\in S,$ and $T_{j}:=2^{-j}T,$
if $j\in S^{c}.$ 
Since the finite linear combinations of the vectors $e_{n}$ from the standard basis of $c_{0}$ forms a dense submanifold satisfying the precribed assumption $\lim_{j\rightarrow \infty}T_{j}x=0,$ $x\in X_{0}$ and the vector $\langle \frac{1}{n}\rangle_{n\in {\mathbb N}}$ is $\lambda$-distributionally unbounded for $ (T_{j})_{j\in {\mathbb N}},$ Corollary \ref{na-dobrorade} implies that 
$ (T_{j})_{j\in {\mathbb N}}$ is densely $\lambda$-distibutionally chaotic. On the other hand, $ (T_{j})_{j\in {\mathbb N}}$ cannot be $\lambda'$-distributionally chaotic since for each $x\in X$ and $\sigma >0$ we have the existence of a finite set $D\subseteq {\mathbb N}$ such that $S^{c} \setminus D \subseteq \{j\in {\mathbb N} : \|T_{j}x\|<\sigma\}$ and therefore $\underline{d}_{1/\lambda'}(\{j\in {\mathbb N} : \|T_{j}x\|<\sigma\})=+\infty.$
\end{example}

We continue by stating the following existence type result closely related with Theorem \ref{primer} and Example \ref{manjoza}:

\begin{thm}\label{primer-primexsimex}
Suppose that $X:=c_{0}({\mathbb N})$ or $X:=l^{p}({\mathbb N})$ for some $p\in [1,\infty).$ Then for each number $\lambda \in (0,1]$ there exists a continuous linear operator $T$ on $X$ satisfying \eqref{doktorbalija},
$\lim_{j\rightarrow \infty}T^{j}x=0$ for some $x\in X$ iff $x=0,$ 
which is
$\lambda$-distributionally chaotic and not $\lambda'$-distributionally chaotic for any $\lambda'\in (0,\lambda).$ 
\end{thm}

\begin{proof}
Let $\lambda \in (0,1]$ and $\lambda'\in (0,\lambda).$ 
As above, we may assume without loss of generality that $X=l^{2}.$ The construction of a weighted forward shift $T\equiv F_{\omega}: 
l^{2} \rightarrow l^{2},$ defined by $F_{\omega} (x_{1}, x_{2}, \cdot \cdot \cdot) \mapsto (0, \omega_{1}x_{1},\omega_{2}x_{2},\cdot \cdot \cdot),$ where the sequence of weights $\omega =(\omega_{k})_{k\in {\mathbb N}}$ consists of sufficiently large blocks of $2$'s of lengths $b_{1}, \ b_{2},\cdot \cdot \cdot,$ and sufficiently large blocks of $(1/2)$'s of lengths $a_{1}, \ a_{2},\cdot \cdot \cdot$
now goes as follows. Set 
$$
a_{0}:=0\ \ ,\ \ a_{n}:=\frac{1}{2}\Bigl[ \lfloor (n+1)^{4/\lambda}\ln (n+1)\rfloor -\lfloor n^{4/\lambda}\ln n\rfloor \Bigr],\quad n\in {\mathbb N}
$$
and
$$
b_{n}:=a_{n-1}+\frac{1}{2}\bigl( 3n^{2}-3n+1\bigr),\quad n\in {\mathbb N}.
$$
Then 
$$
2\bigl( b_{1}+\cdot \cdot \cdot +b_{n}\bigr)=\lfloor n^{4/\lambda}\ln n \rfloor +n^{3},\quad n\in {\mathbb N}
$$
and
$$
2\bigl( a_{1}+\cdot \cdot \cdot +a_{n}\bigr)=\lfloor (n+1)^{4/\lambda}\ln (n+1) \rfloor ,\quad n\in {\mathbb N}.
$$
Arguing as in the previous theorem, we get that the existence
of an integer $n\in {\mathbb N}$ such that  $N\in [b_{n}+\sum_{s=1}^{n
-1}(a_{s}+b_{s}),\sum_{s=1}^{n
}(a_{s}+b_{s})]$ implies
\begin{align*}
\frac{1}{N}&\sum_{j=1}^{N}\bigl \|T^{j}e_{1}\bigr\| \geq \frac{2^{b_{n}-a_{n-1}}
-1}{2na_{n}},
\end{align*}
as well as that the existence 
of an integer $n\in {\mathbb N}$ such that  $N\in [\sum_{s=1}^{n
-1}(a_{s}+b_{s}),b_{n}+\sum_{s=1}^{n
-1}(a_{s}+b_{s})]$ implies
\begin{align*}
\frac{1}{N}&\sum_{j=1}^{N}\bigl \|T^{j}e_{1}\bigr\| \geq   \frac{1}{2nb_{n}}\sum_{j=0}^{b_{n-1}-a_{n-2}}2^{j}\geq \frac{2^{b_{n-1}-a_{n-2}}
-1}{2nb_{n}}.
\end{align*}
Since $\lim_{j\rightarrow \infty}T^{j}x=0$ for some $x\in X$ iff $x=0$ and
$$
\lim_{n\rightarrow +\infty}\frac{2^{b_{n}-a_{n-1}}}{2n\min(a_{n},b_{n-1})}=+\infty,
$$
the foregoing arguments show that
it is sufficient to show that
the following holds for each integer $k\in {\mathbb Z}:$
\begin{itemize}
\item[1.] $\underline{d}_{1/\lambda'}\bigl( \{  j\in {\mathbb N}  : \|T^{j}e_{1}\|>2^{k} \} \bigr)=+\infty ,$ 
\item[2.] $\underline{d}_{1/\lambda}\bigl( \{  j\in {\mathbb N}  : \|T^{j}e_{1}\|>2^{k} \} \bigr)=0$ and $\underline{d}_{1/\lambda}\bigl( \{  j\in {\mathbb N}  : \|T^{j}e_{1}\|<2^{k} \} \bigr)=0.$
\end{itemize}
 But, [1.-2.] can be verified to be true on the basis of our consideration from Example \ref{manjoza}.
\end{proof}

\begin{rem}\label{inverz-lambdainr}
Arguing as in \cite[Remark 26]{2018JMMA} and the proof of above theorem, we can prove the existence of an invertible continuous linear operator $T$ on $X:=c_{0}({\mathbb Z})$ or $X:=l^{p}({\mathbb Z})$ for some $p\in [1,\infty),$ satisfying the all required properties from Theorem \ref{primer-primexsimex}.
\end{rem}

Now we state the following useful proposition:

\begin{prop}\label{propa}
Suppose that $T$ is a Banach space, $T\in L(X)$ and \eqref{doktorbalija} holds.
Then we have:
\begin{itemize}
\item[(i)] Let $(m_{n}) \in {\mathrm R}$ and there exist $n_{0}\in {\mathbb N}$ such that $m_{n}\geq n\|T\|^{n},$ $n\in {\mathbb N},$ $n\geq n_{0}.$ Then $T$ cannot be $m_{n}$-distributionally chaotic (of type $1$) of types $2,$ $2_{Bd}$
or reiteratively $m_{n}$-distributionally chaotic of types $1+,$
$1^{+}.$ 
\item[(ii)]  Let for each $c>0$ there exist $n_{0}\in {\mathbb N}$ such that $(m_{n}) \in {\mathrm R}$ and $m_{n}\geq n\|T\|^{cn},$ $n\in {\mathbb N},$ $n\geq n_{0}.$ Then $T$ cannot be reiteratively $m_{n}$-distributionally chaotic of type $1.$
\end{itemize}
\end{prop}

\begin{proof}
We will prove only (i), for $m_{n}$-distributional chaos. Assume to the contrary that $T$ is $m_{n}$-distributionally chaotic and \eqref{doktorbalija} holds. Then $\|T\|>1$ and it is not difficult to see that there exist a non-zero vector $x\in X$ and a strictly increasing sequence $(n_{k})$ of positive integers satisfying that there exists $k_{0}\in {\mathbb N}$ such that
$k_{0}\geq n_{0}$ and the interval $[1,m_{n_{k}}]$ contains at least $m_{n_{k}}-n_{k}$   
integers $j\in {\mathbb N}$ for which $\|T^{j}x\|\leq 1$ ($k\geq k_{0}$). For any other integer $j\in [1, m_{n_{k}}],$ it is almost trivial to show that $\|T^{j}x\|\leq \|T\|^{n_{k}},$ so that 
\begin{align*}
\frac{1}{m_{n_{k}}}\sum_{i=1}^{m_{n_{k}}}\bigl \|T^{j}x\bigr\| \leq \frac{1}{m_{n_{k}}}\Bigl[ m_{n_{k}}-n_{k} +n_{k}\bigl\|T\bigr\|^{n_{k}}\Bigr]\leq 2,
\end{align*}
for $k\geq k_{0}.$ This clearly contradicts \eqref{doktorbalija}, so that $T$ is not $m_{n}$-distributionally chaotic. 
\end{proof}

As already shown in Theorem \ref{primer-primexsimex}, the notions of $\lambda$-distributional chaos and $\lambda'$-distributional chaos do not coincide for orbits of linear continuous operators on Banach spaces ($\lambda, \ \lambda'\in (0,1],$ $\lambda \neq \lambda'$). Using Proposition \ref{propa}, it is almost trivial to show that the notions being $\lambda$-distributionally chaotic for each number $\lambda \in (0,1]$ and being
$m_{n}$-distributionally chaotic for each sequence $(m_{n}) \in {\mathrm R}$ do not coincide for orbits of linear continuous operators on Banach spaces, as well:

\begin{example}\label{primeremen}
Consider
the weighted forward shift $ T\equiv F_{\omega}: 
l^{2} \rightarrow l^{2}$ from Theorem \ref{primer}. Then we know that $T$ is $\lambda$-distributionally chaotic for each number $\lambda \in (0,1]$ and \eqref{doktorbalija} holds.
Let $(m_{n}) \in {\mathrm R}$ and $m_{n}\geq n\|T\|^{n},$ $n\in {\mathbb N}.$ Due to Proposition \ref{propa}, $T$ cannot be $(m_{n})$-distributionally chaotic (of type $1$) of types $2,$ $2_{Bd}$
or reiteratively distributionally chaotic of types $1+,$
$1^{+}.$  
\end{example}

Keeping in mind Theorem \ref{primer} and Proposition \ref{propa}, it is quite natural to ask whether there exists a weighted forward shift operator $T\equiv F_{\omega}$ satisfying \eqref{doktorbalija} and the additional property that $T$ is $m_{n}$-distributionally chaotic for some sequence $(m_{n}) \in {\mathrm R}$ growing ultrapolynomially at plus infinity. As the next theorem shows, the answer is affirmative:

\begin{thm}\label{pripazise-pricuvajse}
Suppose that $X:=c_{0}({\mathbb N})$ or $X:=l^{p}({\mathbb N})$ for some $p\in [1,\infty).$ Then for each numbers $a>0$ and $b \in (0,1)$ there exists a weighted forward shift operator $T$ on $X$ satisfying \eqref{doktorbalija},
$\lim_{j\rightarrow \infty}T^{j}x=0$ for some $x\in X$ iff $x=0,$ which is $(2^{an^{b}})$-distributionally chaotic.
\end{thm}

\begin{proof}
Without loss of generality, we may assume that $a=1$ and $X=l^{2}.$ We construct a weighted forward shift $T\equiv F_{\omega}: 
l^{2} \rightarrow l^{2},$ defined by $F_{\omega} (x_{1}, x_{2}, \cdot \cdot \cdot) \mapsto (0, \omega_{1}x_{1},\omega_{2}x_{2},\cdot \cdot \cdot),$ where the sequence of weights $\omega =(\omega_{k})_{k\in {\mathbb N}}$ consists of sufficiently large blocks of $2$'s of lengths $b_{1}, \ b_{2},\cdot \cdot \cdot,$ and sufficiently large blocks of $(1/2)$'s of lengths $a_{1}, \ a_{2},\cdot \cdot \cdot$
as done below. 
First of all, let
us recall that for each number $n\in {\mathbb Z}$ there exist two finite subsets $A,\ A'\subseteq {\mathbb N}$ and two numbers $n_{0},\ n_{0}'\in {\mathbb N}$ 
such that $\{j\in {\mathbb N} : \|T^{j}e_{1}\| \leq 2^{n}\} = A\cup \bigcup_{s\geq n_{0}}[2(b_{1}+b_{2}+\cdot \cdot \cdot +b_{s})-n,2(a_{1}+a_{2}+\cdot \cdot \cdot +a_{s})+n]$
and
 $\{j\in {\mathbb N} : \|T^{j}e_{1}\| \geq 2^{n}\} = A'\cup \bigcup_{s\geq n_{0}'}[2(a_{1}+a_{2}+\cdot \cdot \cdot +a_{s})+n,2(b_{1}+b_{2}+\cdot \cdot \cdot +b_{s+1})-n].$ Keeping in mind this fact as well as the final part of proof of Theorem \ref{primer}, the points [1.-2.],
it suffices to construct two strictly increasing sequences $(a_{n})$ and $(b_{n})$ of natural numbers satisfying that there exists an integer $n_{1}\in {\mathbb N}$ such that the following holds:
\begin{itemize}
\item[(i)] $b_{n}<a_{n}<b_{n+1}<a_{n+1}<\cdot \cdot \cdot$ for $n\geq n_{1};$
\item[(ii)] $b_{n+1}\geq 2a_{n}$ for $n\geq n_{1};$
\item[(iii)] $\lim_{n\rightarrow \infty}\frac{2^{b_{n}/2}}{nb_{n+1}}=+\infty;$
\item[(iv)] $\lim_{k\rightarrow \infty}\frac{a_{1}+\cdot \cdot \cdot +a_{4^{k}+k}}{(\log_{2}2b_{4^{k}+k+1})^{1/b}}=0;$
\item[(v)] $\lim_{k\rightarrow \infty}\frac{b_{1}+\cdot \cdot \cdot +b_{4^{k}+1}}{(\log_{2}2a_{4^{k}+1})^{1/b}}=0.$
\end{itemize}
We define the sequence $(b_{n})$ inductively by $b_{1}:=2$ and $b_{n+1}:=\lfloor n^{-2}2^{b_{n}/2} \rfloor$ for $n\geq 1.$ Then (iii) holds and, since $b\in (0,1)$, it is checked at once that for each finite number $\sigma>0$ we have
\begin{align}\label{mangupe-derane}
\lim_{n\rightarrow \infty}\frac{n^{\sigma}b_{n}}{(\log_{2}2b_{n+1})^{1/b}}=0.
\end{align}
Next, we define the sequence $(a_{n})$ in the following way: If there exists $k\in {\mathbb N}$ such that $n=4^{k}+1,$ then we set 
$$
a_{n}:=\Bigl\lceil 2^{[k(b_{1}+\cdot \cdot \cdot+b_{4^{k}+1})]^{b}-1}\Bigr\rceil;
$$ 
otherwise, we set $a_{n}:=b_{n}+1.$ Then 
$$
\frac{b_{1}+\cdot \cdot \cdot +b_{4^{k}+1}}{(\log_{2}2a_{4^{k}+1})^{1/b}}=1/k,\quad k\in {\mathbb N}
$$ 
and (v) trivially holds. Furthermore, the requirements (i)-(ii) can be also trivially verified due to the condition $b\in (0,1).$ To prove (iii), observe that there exist sufficiently large finite numbers $k_{0}\in {\mathbb N}$ and $d>0$ such that for all integers $k\geq k_{0}$ one has:
\begin{align*}
\frac{a_{1}+\cdot \cdot \cdot +a_{4^{k}+k}}{(\log_{2}2b_{4^{k}+k+1})^{1/b}}
\leq \frac{k_{0}d+(4^{k}+k)(b_{4^{k}+k}+1)}{(\log_{2}2b_{4^{k}+k+1})^{1/b}}.
\end{align*} 
Using this estimate, the requirement (iii) follows by applying \eqref{mangupe-derane}. This completes the proof of theorem.
\end{proof}

We continue by stating the following extension of \cite[Proposition 1.1]{bk}, where we observe that the ($m_{n}$-)reiterative distributional chaos of type $0$ for orbits of linear continuous operators on Banach spaces is equivalent with the Li-Yorke chaos. Furthermore, we reexamine the same question for operators on Fr\' echet spaces and obtain only some partial results in this direction:

\begin{prop}\label{ekvivalentno}
Suppose that $T\in L(X)$ and $(m_{n}) \in {\mathrm R}$. Consider the following statements: 
\begin{itemize}
\item[(a)]$T$ is (densely) Li-Yorke chaotic.
\item[(b)] $T$ is (densely) $m_{n}$-reiteratively distributionally chaotic of type $0.$
\item[(c)] $T$ is (densely) reiteratively distributionally chaotic of type $0.$
\end{itemize}
Then we have the following:
\begin{itemize}
\item[(i)] If $X$ is a Banach space, then the statements \emph{(a),} \emph{(b)} and \emph{(c)} are mutually equivalent.
\item[(ii)] If $X$ is a Fr\'echet space, then we have \emph{(b)} $\Rightarrow$ \emph{(c)} $\Rightarrow$ \emph{(a)}. Moreover, the validity of \emph{(a)} implies that there exist an uncountable set $S\subseteq X,$ a positive integer $m\in {\mathbb N}$ and a strictly increasing sequence $(j_{k})$ in ${\mathbb N}$ such that
$\lim_{j\rightarrow \infty}p_{m}(T^{j_{k}}x)=+\infty$ as well as that for each number $s>0$ there exist a finite number $F(s)>0$ and an integer $k_{0}=k_{0}(s)\in {\mathbb N}$ such that for each pair $y,\ z\in S$ of distinct points 
we have $BG_{y,z,m_{n}}\equiv 0$ and
\begin{align}\label{zubice}
\Bigl\{ j\in {\mathbb N} : d(T^{j}y,T^{j}z)<F(s) \Bigr\} \cap \bigl[j_{k}-\lceil m_{s}\rceil,j_{k}\bigr]=\emptyset \ \ \mbox{ for all } \ \ k\geq k_{0}. 
\end{align}
\end{itemize}
\end{prop}

\begin{proof}
The implications (b) $\Rightarrow$ (c) $\Rightarrow$ (a) are trivial and for the proof of (i) we only need to show that (a) implies (b). Towards this end, let us
recall that $T$ is Li-Yorke chaotic iff $T$ admits an irregular vector $x\in X$ for $T,$ i.e., the vector $x\in X$ such that the sequence $(T^{j}x)$ 
is unbounded and has a subsequence converging to zero.  
In the Banach space setting, this means  
that there exist 
two strictly increasing sequences $(l_{k})$ in ${\mathbb N}$ and $(j_{k})$ in ${\mathbb N}$  such that $\lim_{k\rightarrow \infty}\|T^{l_{k}}x\|=0$ and
$\lim_{k\rightarrow \infty}\|T^{j_{k}}x\|=+\infty.$ Let $s>0$ be fixed, and let $\sigma>0$ be arbitrarily chosen. Set $S:=span\{x\}.$ By definitions of $\underline{Bd}_{l,m_{n}}(\cdot)$ and $m_{n}$-reiterative distributional chaos of type $0,$
it suffices to prove that
\begin{align}\label{dziber}
\liminf_{n\rightarrow \infty}\frac{\Bigl| \bigl\{j\in {\mathbb N} : \|T^{j}x\|>\sigma \bigr\} \cap [n+1,n+m_{s}] \Bigr| }{s}=0
\end{align}
and
\begin{align}\label{dziber1}
\liminf_{n\rightarrow \infty}\frac{\Bigl| \bigl\{j\in {\mathbb N} : \|T^{j}x\|<\sigma \bigr\} \cap [n+1,n+m_{s}] \Bigr| }{s}=0.
\end{align}
It is clear that there exist two strictly increasing sequences of positive integers $(l_{k}')$ and $(j_{k}')$ with unbounded differences such that 
$\| T^{l_{k}'}x\|<\sigma(2+\|T\|)^{-k^{2}-m_{s}-1}/2$ and $\| T^{j_{k}'}x\|>2\sigma(2+\|T\|)^{k^{2}+m_{s}+1}$ for all $k\in {\mathbb N}.$ 
An elementary line of reasoning shows that the sets $\{j\in {\mathbb N} : \|T^{j}x\|>\sigma \} \cap [l_{k}',l_{k}'+\lceil m_{s} \rceil] $ and 
$\{j\in {\mathbb N} : \|T^{j}x\|<\sigma \} \cap [j_{k}'-\lceil m_{s}\rceil , j_{k}'] $ are empty, finishing the proofs of \eqref{dziber}-\eqref{dziber1} and (i). 
In the Fr\' echet space setting,
we have that there exist
an integer $m\in {\mathbb N}$ and a strictly increasing sequence $(j_{k})$ in ${\mathbb N}$ such that
$\lim_{k\rightarrow \infty}p_{m}(T^{j_{k}}x)=+\infty.$ 
Set, as above, $S:=span\{x\}$
and fix numbers $s>0,$ $\alpha \in {\mathbb K}$ and $\beta \in {\mathbb K}$ ($\alpha \neq \beta$). Due to the continuity of $T,$ we can find two strictly increasing sequences $(c_{j})$ and $(r_{j})$ of positive integers such that $p_{r_{j}}(Ty)\leq c_{j}p_{r_{j+1}}(y)$ for all $y\in X$ and $j\in {\mathbb N}.$ Define
$$
F(s):=\sum_{a=\lceil m_{s} \rceil +m}^{\infty}2^{-1-r_{a+1}}.
$$
It remains to be proved that $BG_{\alpha x,\beta x,m_{n}}\equiv 0$ and \eqref{zubice} holds with $y=\alpha x$ and $z=\beta x.$ For the first equality, choose $\epsilon>0$ arbitrarily. Suppose that $\lim_{k\rightarrow \infty}T^{l_{k}}x=0$ and $j\in [l_{k},l_{k}+\lceil m_{s} \rceil]$ for some $k\in{\mathbb N}.$
It is clear that there exists an integer $a\in {\mathbb N}$ such that $\sum_{l> r_{a}}2^{-l}\leq \epsilon/2.$ Then we have 
\begin{align*}
d\bigl(T^{j}\alpha x ,T^{j}\beta x\bigr)& \leq \sum_{l=1}^{r_{a}}\frac{1}{2^{l}}\frac{p_{l}\bigl(T^{j}(\alpha-\beta)x\bigr)}{1+p_{l}\bigl(T^{j}(\alpha-\beta)x\bigr)}+\sum_{l>r_{a}}\frac{1}{2^{l}}\frac{p_{l}\bigl(T^{j}(\alpha-\beta)x\bigr)}{1+p_{l}\bigl(T^{j}(\alpha-\beta)x\bigr)}
\\ & \leq \sum_{l=1}^{r_{a}}\frac{1}{2^{l}}\frac{|\alpha-\beta| p_{r_{a}}(T^{l}x)}{1+|\alpha-\beta| p_{r_{a}}(T^{l}x)}+\sum_{l>r_{a}}\frac{1}{2^{l}}
\\ & \leq \frac{|\alpha-\beta| p_{r_{a}}(T^{l}x)}{1+|\alpha-\beta| p_{r_{a}}(T^{l}x)}+\frac{\epsilon}{2}
\\ & \leq \frac{|\alpha-\beta| c_{1}c_{2}\cdot \cdot \cdot c_{\lceil m_{s} \rceil}p_{r_{a+\lceil m_{s} \rceil}}(T^{l_{k}}x)}{ 1+|\alpha -\beta|c_{1}c_{2}\cdot \cdot \cdot c_{\lceil m_{s} \rceil}p_{r_{a+\lceil m_{s} \rceil}}(T^{l_{k}}x)}+\frac{\epsilon}{2},
\end{align*}
so that there exists a sufficiently large number $k_{0}'=k_{0}'(s)\in {\mathbb N}$ such that for each integer $k\in {\mathbb N}$ with $k\geq k_{0}$ and for each integer $j\in [l_{k},l_{k}+\lceil m_{s} \rceil]$ we have $d(T^{j}\alpha x ,T^{j}\beta x)\leq \epsilon.$ Hence, $BG_{\alpha x, \beta x,m_{n}}\equiv 0.$
Suppose now that $k\in {\mathbb N},$ $j\in {\mathbb N} \cap [j_{k}-\lceil m_{s}\rceil,j_{k}]$ and $a\geq \lceil m_{s} \rceil +m.$ Then it is easy to see that $p_{r_{a}}(T^{j}x)\geq [c_{a-1}c_{a-2}\cdot \cdot \cdot c_{a-(j_{k}-j)}]^{-1}p_{r_{a-(j_{k}-j)}}(T^{j_{k}}x),$ so that
\begin{align*}
d\bigl( \alpha T^{j}x,\beta T^{j}x \bigr)& \geq \sum_{a=\lceil m_{s} \rceil +m}^{\infty}2^{-r_{a+1}}\frac{|\alpha-\beta|p_{r_{a}}(T^{j}x)}{1+|\alpha-\beta|p_{r_{a}}(T^{j}x)}
\\ & \geq \sum_{a=\lceil m_{s} \rceil +m}^{\infty}2^{-r_{a+1}}\frac{|\alpha-\beta|p_{r_{a-(j_{k}-j)}}(T^{j_{k}}x)}{c_{a-1}c_{a-2}\cdot \cdot \cdot c_{a-(j_{k}-j)}+|\alpha-\beta|p_{r_{a-(j_{k}-j)}}(T^{j_{k}}x)}
\\ & \geq \sum_{a=\lceil m_{s} \rceil +m}^{\infty}2^{-r_{a+1}}\frac{|\alpha-\beta|p_{r_{a-\lceil m_{s} \rceil}}(T^{j_{k}}x)}{c_{a-1}c_{a-2}\cdot \cdot \cdot c_{a-(j_{k}-j)}+|\alpha-\beta|p_{r_{a-\lceil m_{s} \rceil}}(T^{j_{k}}x)}
\\ & \geq \sum_{a=\lceil m_{s} \rceil +m}^{\infty}2^{-r_{a+1}}\frac{|\alpha-\beta|p_{r_{a-\lceil m_{s} \rceil}}(T^{j_{k}}x)}{c_{a-1}c_{a-2}\cdot \cdot \cdot c_{a-\lceil m_{s}\rceil}+|\alpha-\beta|p_{r_{a-\lceil m_{s} \rceil}}(T^{j_{k}}x)}
\\ & \geq \sum_{a=\lceil m_{s} \rceil +m}^{\infty}2^{-r_{a+1}}\frac{|\alpha-\beta|p_{m}(T^{j_{k}}x)}{c_{a-1}c_{a-2}\cdot \cdot \cdot c_{a-\lceil m_{s}\rceil}+|\alpha-\beta|p_{m}(T^{j_{k}}x)}.
\end{align*}
Since $\lim_{k\rightarrow \infty}p_{m}(T^{j_{k}}x)=+\infty,$ the above implies the existence of a sufficiently large number $k_{0}=k_{0}(s)\in {\mathbb N}$ such that for each $k\in {\mathbb N}$ with $k\geq k_{0}$,  $j\in {\mathbb N} \cap [j_{k}-\lceil m_{s}\rceil,j_{k}]$ and $a\geq \lceil m_{s} \rceil +m,
$ we have $d( \alpha T^{j}x,\beta T^{j}x ) \geq F(s).$ This gives \eqref{zubice} and completes the proof of proposition.
\end{proof}

\begin{rem}\label{pripazise}
By the proof of above proposition, we have the following equivalence relations in Banach spaces:
\begin{itemize}
\item[(i)] An element $x\in X$ is Li-Yorke unbounded for $T$ (i.e., $\inf_{j\in {\mathbb N}}\|T^{j}x\|=0$) iff for each sequence $(m_{n}) \in {\mathrm R}$ there exists an infinite set $B\subseteq {\mathbb N}$ such that $\lim_{j\in B}\|T^{j}x\|=+\infty$ and $\underline{Bd}_{l;m_{n}}(B^{c})=0.$
\item[(ii)] An element $x\in X$ is Li-Yorke near to zero for $T$ (i.e., $\sup_{j\in {\mathbb N}}\|T^{j}x\|=+\infty$) iff for each sequence $(m_{n}) \in {\mathrm R}$ there exists an infinite set $B\subseteq {\mathbb N}$ such that $\lim_{j\in B}\|T^{j}x\|=0$ and $\underline{Bd}_{l;m_{n}}(B^{c})=0.$
\end{itemize} 
\end{rem}

For operators on Banach spaces, it is well known that the chaos DC2 is equivalent with chaos DC1 (see \cite[Theorem 2]{2018JMMA}).
This basically follows from an application of \cite[Proposition 8]{2013JFA} and the first problem that we want to announce is a question whether the condition (i)' in the formulation of this proposition implies the condition (i) for operators in Fr\' echet spaces: 

\begin{prob}\label{prob1}
 Let $X$ be a Fr\' echet space and let $T\in L(X).$ Consider the following conditions:
\begin{itemize}
\item[(i)':] There exist $\epsilon>0,$ a sequence $(y_{k}) \in X$ and an increasing sequence $(N_{k})$ in ${\mathbb N}$ such that $\lim_{k\rightarrow \infty}y_{k}=0$ and
$$
card\Bigl(\bigl\{ 1\leq j\leq N_{k} : d(T^{j}y_{k},0) >\epsilon \bigr\}\Bigr) \geq \epsilon N_{k},\quad k\in {\mathbb N}.
$$
\item[(i):] There exist $\epsilon>0,$ a sequence $(y_{k}) \in X$ and an increasing sequence $(N_{k})$ in ${\mathbb N}$ such that $\lim_{k\rightarrow \infty}y_{k}=0$ and
$$
\lim_{k\rightarrow \infty}\frac{1}{N_{k}}card\Bigl(\bigl\{ 1\leq j\leq N_{k} : d(T^{j}y_{k},0) >\epsilon \bigr\}\Bigr)=1.
$$
\end{itemize}
Is it true that
\emph{(i)'} implies \emph{(i)}?
\end{prob}

Now we would like to raise the following issues closely linked with Problem \ref{prob1}: 

\begin{prob}\label{prob2}
Let $X$ be a Fr\' echet space and let $T\in L(X).$ Is it true that:
\begin{itemize}
\item[(i)] $T$ satisfies \emph{(DC2)} iff $T$ satisfies \emph{(DC1)}?
\item[(ii)] $T$ is reiteratively distributionally chaotic of type $2-$ iff $T$ is reiteratively distributionally chaotic of type $2$?
\end{itemize}
\end{prob}

To the best knowledge of the author, it is still unknown whether the answers to Problem \ref{prob1} and Problem \ref{prob2} are affirmative for certain classes of shift operators in 
Fr\'echet sequence spaces and certain classes of composition operators in Fr\'echet function spaces, and whether a counterexample of such an operator not fulfilling the equivalence relation (a) $\Leftrightarrow$ (b) $\Leftrightarrow$ (c) of Proposition \ref{ekvivalentno} really exists.

\section{Reiterative $m_{n}$-distributionally irregular vectors of type $s$ and reiterative $m_{n}$-distributionally irregular manifolds of type $s$}\label{totijemarko}

In this section, we investigate various notions of
(reiterative) $m_{n}$-distributionally irregular vectors of type $s$ and (reiterative) $m_{n}$-distributionally irregular manifolds of type $s.$
We start by introducing the following definition:

\begin{defn}\label{nizovi-fgh}
Suppose that for each $j\in {\mathbb N},$ $T_{j} : D(T_{j}) \subseteq X \rightarrow Y$ is a linear operator. Then we say that:
\begin{itemize}
\item[(i)] $x$ is ({\it reiteratively}) $m_{n}${\it -distributionally
near to $0$} for $(T_{j})_{j\in {\mathbb N}}$
iff there is $A\subseteq {\mathbb N}$ such that ($\underline{Bd}_{l;m_{n}}(A^{c})=0$) $\underline{d}_{m_{n}}(A^{c})=0$ and 
$\lim_{j\in A,j\rightarrow \infty}T_{j}x=0;$
\item[(ii)] $x$ is ({\it reiteratively}) $m_{n}${\it-distributionally $m$-unbounded} for $(T_{j})_{j\in {\mathbb N}}$ iff there is $B\subseteq
{\mathbb N}$ such that  ($\underline{Bd}_{l;m_{n}}(B^{c})=0$) $\underline{d}_{m_{n}}(B^{c})=0$  and
$\lim_{j\in
B,j\rightarrow \infty}p_{m}^{Y}(T_{j}x)=\infty ;$ $x$ is said to be ({\it reiteratively}) $m_{n}${\it-distributionally unbounded} for $(T_{j})_{j\in {\mathbb N}}$ iff there
exists $q\in {\mathbb N}$ such that $x$ is (reiteratively) $m_{n}$-distributionally $q$-unbounded for $(T_{j})_{j\in {\mathbb N}}$
(if $Y$ is a Banach space, this simply means that  $\lim_{j\in
B,j\rightarrow \infty}\|T_{j}x\|_{Y}=\infty ).$
\end{itemize}
\end{defn}

In the following two definitions, we introduce separately the notions of $m_{n}$-distributionally irregular vectors of type $s\in \{1,2,2\frac{1}{2},3,2_{Bd}
\}$ and reiteratively $m_{n}$-distributionally irregular vectors of type $s\in \{0,1,1+,2,2_{Bd},2-
\}:$

\begin{defn}\label{-ira}
Suppose that for each $j\in {\mathbb N},$ $T_{j} : D(T_{j}) \subseteq X \rightarrow Y$ is a linear operator. Then we say that:
\begin{itemize}
\item[(i)] $x$ is an {\it $m_{n}$-distributionally irregular vector} for
$(T_{j})_{j\in {\mathbb N}}$ iff $x$ is $m_{n}$-distributionally
near to $0$ for $(T_{j})_{j\in {\mathbb N}}$ and $x$ is $m_{n}$-distributionally
unbounded for $(T_{j})_{j\in {\mathbb N}}$; 
\item[(ii)] $x$ is an $m_{n}${\it -distributionally irregular vector of type $1$} for
$(T_{j})_{j\in {\mathbb N}}$ iff $x$ is $m_{n}$-distributionally
near to $0$ for $(T_{j})_{j\in {\mathbb N}}$ and $F_{x,0,m_{n}}(\sigma) =0$
for some $\sigma > 0;$
\item[(iii)] $x$ is an $m_{n}${\it -distributionally irregular vector of type $2$} for
$(T_{j})_{j\in {\mathbb N}}$ iff there exists a finite number $\sigma>0$ such that 
$G_{x,0,m_{n}} \equiv 0$ and $I_{x,0,m_{n}}(\sigma) >0;$
\item[(iv)] $x$ is an {\it $m_{n}$-distributionally irregular vector of type $2\frac{1}{2}$} for
$(T_{j})_{j\in {\mathbb N}}$ iff there exist real numbers $c > 0$ and $r > 0$ such that
 $F_{x,0,m_{n}}(\delta) < c < H_{x,0,m_{n}}(\delta)$
 for all $ 0 < \delta < r;$
\item[(v)] $x$ is an $m_{n}${\it-distributionally irregular vector of type $3$} for
$(T_{j})_{j\in {\mathbb N}}$ iff there exist real numbers $b>a>  0$ and $c>0$ such that
 $F_{x,0,m_{n}}(\delta) < c < H_{x,0,m_{n}}(\delta)$
 for all $ a \leq \delta \leq b;$
\item[(vi)] $x$ is an $m_{n}${\it-distributionally irregular vector of type $2_{Bd}$} for
$(T_{j})_{j\in {\mathbb N}}$ iff $x$ is $m_{n}$-distributionally near to zero and 
there exists a finite number $\sigma>0$ such that
$BI_{x,0,m_{n}}(\sigma) >0.$
\end{itemize}
\end{defn}

\begin{defn}\label{reiterative-ira}
Suppose that for each $j\in {\mathbb N},$ $T_{j} : D(T_{j}) \subseteq X \rightarrow Y$ is a linear operator. Then we say that:
\begin{itemize}
\item[(i)] $x$ is a {\it reiteratively $m_{n}$-distributionally irregular vector of type $0$} for
$(T_{j})_{j\in {\mathbb N}}$ iff $x$ is reiteratively $m_{n}$-distributionally
near to $0$ for $(T_{j})_{j\in {\mathbb N}}$ and $x$ is reiteratively $m_{n}$-distributionally
unbounded for $(T_{j})_{j\in {\mathbb N}}$; 
\item[(ii)] $x$ is a {\it reiteratively} $m_{n}${\it -distributionally irregular vector of type $1$} for
$(T_{j})_{j\in {\mathbb N}}$ iff there exist two finite numbers $c \in (0,\liminf_{n\rightarrow \infty}\frac{m_{n}}{n})$ and $r>0$ such that  $G_{x,0,m_{n}}(\delta)\leq c$
for $0<\delta<r,$
and $x$ is reiteratively $m_{n}$-distributionally unbounded for $(T_{j})_{j\in {\mathbb N}};$
\item[(iii)] $x$ is a {\it reiteratively $m_{n}$-distributionally irregular vector of type $1+$} ($1^{+}$) for
$(T_{j})_{j\in {\mathbb N}}$ iff $x$ is $m_{n}$-distributionally near to zero and $x$ is reiteratively $m_{n}$-distributionally unbounded for $(T_{j})_{j\in {\mathbb N}};$
\item[(iv)] $x$ is a {\it reiteratively} $m_{n}${\it-distributionally irregular vector of type $2$} for
$(T_{j})_{j\in {\mathbb N}}$ iff $x$ is reiteratively $m_{n}$-distributionally near to zero and 
there exists $\sigma>0$ such that $I_{x,0,m_{n}}(\sigma) >0;$
\item[(v)] $x$ is a {\it reiteratively} $m_{n}${\it-distributionally irregular vector of type $2_{Bd}$} for
$(T_{j})_{j\in {\mathbb N}}$ iff $x$ is reiteratively $m_{n}$-distributionally near to zero and there exists a finite number $\sigma>0$ such that
$BI_{x,0,m_{n}}(\sigma) >0;$
\item[(vi)] $x$ is a {\it reiteratively $m_{n}$-distributionally irregular vector of type $2-$} for
$(T_{j})_{j\in {\mathbb N}}$ iff $x$ is reiteratively $m_{n}$-distributionally near to zero and $x$ is $m_{n}$-distributionally unbounded for $(T_{j})_{j\in {\mathbb N}}.$
\end{itemize}
\end{defn}

We will employ the following notion, as well:

\begin{defn}\label{idiotisen-obicno}
Suppose that for each $j\in {\mathbb N},$ $T_{j} : D(T_{j}) \subseteq X \rightarrow Y$ is a linear operator.
Let $\{0\} \neq X'  \subseteq \tilde{X}$ be a linear manifold and let $s\in \{1,2,2\frac{1}{2},3,2_{Bd}\}$. Then
we say that:
\begin{itemize}
\item[(i)]
$X'$ is ({\it dense,} if $X'$ is dense in $\tilde{X}$) $\tilde{X}_{m_{n}}${\it -distributionally
irregular manifold} ({\it of type $s$})  for $(T_{j})_{j\in {\mathbb N}}$
iff any element $x\in (X' \cap
\bigcap_{j=1}^{\infty}D(T_{j})) \setminus \{0\}$ is an
$m_{n}$-distributionally irregular vector (of type $s$) for
$(T_{j})_{j\in {\mathbb N}};$
\item[(ii)] 
$X'$ is a  ({\it dense,} if $X'$ is dense in $\tilde{X}$) {\it  uniformly $\tilde{X}_{m_{n}}$-distributionally
irregular manifold} for $(T_{j})_{j\in {\mathbb N}}$
if, in addition to the above, there exists $m\in {\mathbb N}$ such that
any vector $x\in (X' \cap
\bigcap_{j=1}^{\infty}D(T_{j})) \setminus \{0\}$ is $m_{n}$-distributionally $m$-unbounded;
\item[(iii)] $X'$ is ({\it dense,} if $X'$ is dense in $\tilde{X}$) $\tilde{X}_{m_{n}}${\it -distributionally
irregular manifold} {\it of type $2_{Bd}+$}  for $(T_{j})_{j\in {\mathbb N}}$
iff $X'$ is (dense, if $X'$ is dense in $\tilde{X}$) $\tilde{X}_{m_{n}}$-distributionally
irregular manifold of type $2_{Bd}$ for $(T_{j})_{j\in {\mathbb N}}$ and the number $\sigma>0$ in Definition \ref{-ira}(iii) is independent of choice of element $x\in (X' \cap
\bigcap_{j=1}^{\infty}D(T_{j})) \setminus \{0\}.$
\end{itemize}
\end{defn}

\begin{defn}\label{idiotisen}
Suppose that for each $j\in {\mathbb N},$ $T_{j} : D(T_{j}) \subseteq X \rightarrow Y$ is a linear operator.
Let $\{0\} \neq X'  \subseteq \tilde{X}$ be a linear manifold and let $s\in \{0,1,1+,1^{+},2,2_{Bd}, 2-\}$. Then
we say that:
\begin{itemize}
\item[(i)]
$X'$ is ({\it dense,} if $X'$ is dense in $\tilde{X}$) {\it reiteratively $\tilde{X}_{m_{n}}$-distributionally
irregular manifold of type $s$}  for $(T_{j})_{j\in {\mathbb N}}$
iff any element $x\in (X' \cap
\bigcap_{j=1}^{\infty}D(T_{j})) \setminus \{0\}$ is a
reiteratively $m_{n}$-distributionally irregular vector of type $s$ for
$(T_{j})_{j\in {\mathbb N}};$
\item[(ii)] If $s\in \{ 0,1,1^{+}\},$ then we say that
$X'$ is a ({\it dense,} if $X'$ is dense in $\tilde{X}$) {\it  uniformly $\tilde{X}_{m_{n}}$-reiteratively distributionally
irregular manifold of type $s$} for $(T_{j})_{j\in {\mathbb N}}$
if, in addition to the above, there exists $m\in {\mathbb N}$ such that
any vector $x\in (X' \cap
\bigcap_{j=1}^{\infty}D(T_{j})) \setminus \{0\}$ is reiteratively $m_{n}$-distributionally $m$-unbounded;
\item[(iii)] If $s=2-,$ then we say that
$X'$ is a ({\it dense,} if $X'$ is dense in $\tilde{X}$) {\it  uniformly $\tilde{X}_{m_{n}}$-reiteratively distributionally
irregular manifold of type $s$} for $(T_{j})_{j\in {\mathbb N}}$
if, in addition to the above, there exists $m\in {\mathbb N}$ such that
any vector $x\in (X' \cap
\bigcap_{j=1}^{\infty}D(T_{j})) \setminus \{0\}$ is $m_{n}$-distributionally $m$-unbounded;
\item[(iv)] If $s\in \{2+,2_{Bd}+\},$ then we say that $X'$ is ({\it dense,} if $X'$ is dense in $\tilde{X}$) {\it reiteratively $\tilde{X}_{m_{n}}$-distributionally
irregular manifold of type $s$}  for $(T_{j})_{j\in {\mathbb N}}$
iff any element $x\in (X' \cap
\bigcap_{j=1}^{\infty}D(T_{j})) \setminus \{0\}$ is a
reiteratively $m_{n}$-distributionally irregular vector of type $s$ for
$(T_{j})_{j\in {\mathbb N}}$ and the number $\sigma>0$ in Definition \ref{reiterative-ira}(iv)-(v) is independent of $x.$
\end{itemize}
\end{defn}

All above notions are accepted for a linear operator $T : D(T) \subseteq X \rightarrow X$
by using the sequence $(T_{j}\equiv
T^{j})_{\in {\mathbb N}}$ for definitions.
Further on, we will accept similar agreements as before. If $\tilde{X}=X$ or $m_{n}\equiv n,$ then we remove the prefixes ``$\tilde{X}$-'' and ``$m_{n}$-'' from the terms and notions. For example, dense uniformly reiteratively distributionally
irregular manifold of type $s$ for $(T_{j})_{j\in {\mathbb N}}$ is dense uniformly $X_{n}$-reiteratively distributionally
irregular manifold of type $s$ for $(T_{j})_{j\in {\mathbb N}}.$
If $m_{n}\equiv n^{1/\lambda}$ for some $\lambda \in (0,1],$ then,
as a special case of the above definitions, we obtain 
the notions of (reiteratively) $\lambda$-distributionally near to zero vectors, (reiteratively) $\lambda$-distributionally ($m$-)unbounded vectors  
and (reiteratively) $\lambda$-distributionally irregular vectors of type $s$ for sequence $(T_{j})_{j\in {\mathbb N}}$ (operator $T$); a similar terminology is used for manifolds.
In the case that $\lambda=1,$ 
then we remove the prefix ``$\lambda$-'' from the terms and notions.

The following statements hold:

\begin{itemize}
\item[A.] 
Using the elementary properties of metric, it can be simply verified that 
$X'$ is a $\tilde{X}_{m_{n}}$-scrambled set for
$(T_{j})_{j\in {\mathbb N}}$ whenever $X'$ is a uniformly reiteratively $\tilde{X}_{m_{n}}$-distributionally
irregular manifold\index{$\tilde{X}$-distributionally
irregular manifold!uniformly} for $(T_{j})_{j\in {\mathbb N}}.$ Furthermore, let $s\in \{1,2\frac{1}{2},3,2_{Bd},2_{Bd}+\};$ then $X'$ is a $\tilde{X}_{m_{n}}$-scrambled set of type $s$ for
$(T_{j})_{j\in {\mathbb N}}$ whenever $X'$ is an $\tilde{X}_{m_{n}}$-distributionally
irregular manifold\index{$\tilde{X}$-distributionally
irregular manifold!uniformly} of type $s$ for $(T_{j})_{j\in {\mathbb N}}.$
On the other hand, we can simply
verify that, if $0\neq x\in \tilde{X}$ is an $m_{n}$-distributionally irregular vector/$m_{n}$-distributionally irregular vector of type $s$ for $(T_{j})_{j\in {\mathbb N}},$
then $X'\equiv span\{x\}$
is a uniformly $\tilde{X}_{m_{n}}$-distributionally irregular manifold for
$(T_{j})_{j\in {\mathbb N}}$ /$\tilde{X}_{m_{n}}$-distributionally irregular manifold of type $s$ for
$(T_{j})_{j\in {\mathbb N}}$ [if $s=3,$ then we need to additionally impose that $Y$ is a Banach space].
\item[B.]
Let $s\in \{0, 1,1^{+},2-\}$.
Using the elementary properties of metric, it can be simply verified that 
$X'$ is a reiteratively $m_{n}$-scrambled set of type $s$ for
$(T_{j})_{j\in {\mathbb N}}$ whenever $X'$ is a uniformly reiteratively $m_{n}$-distributionally
irregular manifold\index{$\tilde{X}$-distributionally
irregular manifold!uniformly} of type $s$ for $(T_{j})_{j\in {\mathbb N}}.$ Furthermore, $X'$ is a reiteratively $m_{n}$-scrambled set of type $2$ ($2_{Bd}$) for
$(T_{j})_{j\in {\mathbb N}}$ whenever $X'$ is a reiteratively $m_{n}$-distributionally
irregular manifold\index{$\tilde{X}$-distributionally
irregular manifold!uniformly} of type $2/2+$ ($2_{Bd}/2_{Bd}+$) for $(T_{j})_{j\in {\mathbb N}}.$
On the other hand, we can simply
verify that, if $0\neq x\in \tilde{X}$ is a reiteratively
$m_{n}$-distributionally irregular vector of type $s$ for $(T_{j})_{j\in {\mathbb N}},$
then $X'\equiv span\{x\}$
is a (uniformly, if $s\notin \{ 2,2_{Bd}\}$) reiteratively  $\tilde{X}_{m_{n}}$-distributionally irregular manifold of type $s$ for
$(T_{j})_{j\in {\mathbb N}}.$ 
\end{itemize}

\subsection{Structural results for $m_{n}$-distributionally irregular vectors}\label{profice}

A fairly complete analysis of (reiteratively) $m_{n}$-distributionally
irregular vectors of type $s$ and corresponding (reiteratively) $m_{n}$-distributionally
irregular manifolds of type $s$ is far from beng easy and trivial. In this subsection, we shall primarily focus our attention on the notions of 
$m_{n}$-distributional chaos, reiterative $m_{n}$-distributional chaos of types $1$, $1+,$ and explain how the results from \cite[Section 2]{2013JFA} can be slightly extended for $m_{n}$-distributional chaos (cf. also \cite{novascience}). 

We start by stating the following extension of \cite[Proposition 7]{2013JFA} and the equivalence relations (i) $\Leftrightarrow$ (ii) $\Leftrightarrow$ (iii) of \cite[Proposition 8]{2013JFA} (concerning this proposition, we feel duty bound to say that the equivalence with (i)' and (ii)' for orbits of a single operator on Banach space is not attainable for $m_{n}$-distributional chaos, as far as we can see):

\begin{prop}\label{beton}
Let $(m_{n})\in {\mathrm R},$ $m\in {\mathbb N}$ and $(T_{j})_{j\in {\mathbb N}}$ be a sequence in $L(X,Y).$
\begin{itemize}
\item[(i)] The following assertions are equivalent:
\begin{itemize}
\item[(a)]
Suppose that there exist a number $\epsilon>0$, a zero sequence $(y_{k})$ in $X$ and a strictly increasing sequence $(N_{k})$ in ${\mathbb N}$ such that
\begin{align}\label{prcko-frcko}
m_{N_{k}}-\Bigl| \bigl\{j\in {\mathbb N} : p_{m}^{Y}(T_{j}y_{k})>\epsilon\bigr\} \cap [1,m_{N_{k}}] \Bigr |\leq \frac{N_{k}}{k},\quad k\in {\mathbb N}.
\end{align}
\item[(b)] The set of $m_{n}$-distributionally $m$-unbounded vectors for $(T_{j})_{j\in {\mathbb N}}$ is non-empty.
\item[(c)] The set of $m_{n}$-distributionally $m$-unbounded vectors for $(T_{j})_{j\in {\mathbb N}}$ is residual in $X.$
\end{itemize}
\item[(ii)] The following assertions are equivalent:
\begin{itemize}
\item[(a)']
Suppose that there exist a number $\epsilon>0$, a zero sequence $(y_{k})$ in $X$ and a strictly increasing sequence $(N_{k})$ in ${\mathbb N}$ such that
\begin{align}\label{prckof-frckof}
m_{N_{k}}-\Bigl| \bigl\{j\in {\mathbb N} : d_{Y}(T_{j}y_{k},0)>\epsilon \bigr\} \cap [1,m_{N_{k}}] \Bigr |\leq \frac{N_{k}}{k},\quad k\in {\mathbb N}.
\end{align}
\item[(b)'] The set of $m_{n}$-distributionally unbounded vectors for $(T_{j})_{j\in {\mathbb N}}$ is non-empty.
\item[(c)'] The set of $m_{n}$-distributionally unbounded vectors for $(T_{j})_{j\in {\mathbb N}}$ is residual in $X.$
\end{itemize}
\end{itemize}
\end{prop}

\begin{proof}
We will only outline the main details for showing the implication (a) $\Rightarrow$ (c) in (i) since the use of arguments contained in the proof of \cite[Proposition 7]{2013JFA} is possible with appropriate modifications described as follows. For each $k\in {\mathbb N},$ we set
$$
M_{k}:=\Bigl\{ x\in X : (\exists n\in {\mathbb N}) \mbox{ s.t. }m_{n}-\Bigl|\bigl\{j\in {\mathbb N} : p_{m}^{Y}(T_{j}x)>k\bigr\} \cap [1,m_{n}]  \Bigr| \leq \frac{n}{k}\Bigr\}.
$$
Then it is very plain to show that for each $k\in {\mathbb N}$ the set $M_{k}$ is open as well as that the set $\bigcap_{k\in{\mathbb N}}M_{k}$ is consisted solely of $m_{n}$-distributionally $m$-unbounded vectors for $(T_{j})_{j\in {\mathbb N}}.$ It remains to be proved that for each $k\in {\mathbb N}$ the set $M_{k}$ is dense. To see this, we can repeat almost literally the arguments contained in the proof of above-mentioned Proposition 7, with the same terminology used and the sets
$$
A:=\Bigl\{1\leq j\leq m_{n} : p_{m}^{Y}(T_{j}u)>\epsilon\Bigr\},
$$
$$
B_{s}:=\Bigl\{ 1\leq j\leq m_{n} : p_{m}^{Y}(T_{j}u_{s})\leq k\Bigr\},\quad s=0,1,\cdot \cdot \cdot, 2k(1+m_{n})-1,
$$
where 
$$
u_{s}:=x+\frac{\delta s u}{2k(1+m_{n})C},\quad s=0,1,\cdot \cdot \cdot, 2k(1+m_{n})-1,
$$
and $n$ is chosen so that
$$
m_{n}-\Bigl| \bigl\{j\in {\mathbb N} : p_{m}^{Y}(T_{j}y_{k})>\epsilon\bigr\} \cap [1,m_{n}] \Bigr |\leq m_{n}-\frac{n}{2k};
$$
cf. \eqref{prcko-frcko}.
\end{proof}

\begin{cor}\label{betona}
Let $\lambda \in (0,1],$ $m\in {\mathbb N}$ and $(T_{j})_{j\in {\mathbb N}}$ be a sequence in $L(X,Y).$
\begin{itemize}
\item[(i)] The following assertions are equivalent:
\begin{itemize}
\item[(a)]
Suppose that there exist a number $\epsilon>0$, a zero sequence $(y_{k})$ in $X$ and a strictly increasing sequence $(N_{k})$ in ${\mathbb N}$ such that
\begin{align*}
N_{k}^{1/\lambda}-\Bigl| \bigl\{j\in {\mathbb N} : p_{m}^{Y}(T_{j}y_{k})>k\bigr\} \cap [1,N_{k}^{1/\lambda}] \Bigr |\leq \frac{N_{k}}{k},\quad k\in {\mathbb N}.
\end{align*}
\item[(b)] The set of $\lambda$-distributionally $m$-unbounded vectors for $(T_{j})_{j\in {\mathbb N}}$ is non-empty.
\item[(c)] The set of $\lambda$-distributionally $m$-unbounded vectors for $(T_{j})_{j\in {\mathbb N}}$ is residual in $X.$
\end{itemize}
\item[(ii)] The following assertions are equivalent:
\begin{itemize}
\item[(a)'] Suppose that there exist a number $\epsilon>0$, a zero sequence $(y_{k})$ in $X$ and a strictly increasing sequence $(N_{k})$ in ${\mathbb N}$ such that
\begin{align}\label{prckofff-frckofff}
N_{k}^{1/\lambda}-\Bigl| \bigl\{j\in {\mathbb N} : d_{Y}(T_{j}y_{k},0)>\epsilon \bigr\} \cap [1,N_{k}^{1/\lambda}] \Bigr |\leq \frac{N_{k}}{k},\quad k\in {\mathbb N}.
\end{align}  
\item[(b)'] The set of $\lambda$-distributionally unbounded vectors for $(T_{j})_{j\in {\mathbb N}}$ is non-empty.
\item[(c)'] The set of $\lambda$-distributionally unbounded vectors for $(T_{j})_{j\in {\mathbb N}}$ is residual in $X.$
\end{itemize}
\end{itemize}
\end{cor}

Considering the sets 
$$
M_{k,m}:=\Bigl\{ x\in X : (\exists n\in {\mathbb N}) \mbox{ s.t. }m_{n}-\Bigl|\bigl\{j\in {\mathbb N} : p_{m}^{Y}(T_{j}x)<1/k\bigr\} \cap [1,m_{n}]  \Bigr| \leq \frac{n}{k}\Bigr\},
$$
we can similarly prove the following extension of \cite[Proposition 9]{2013JFA}:

\begin{prop}\label{nemoj}
Suppose that $(T_{j})_{j\in {\mathbb N}}$ is a sequence in $L(X,Y).$ If the set of those vectors $x\in X$ for which there exists a set $B\subseteq {\mathbb N}$ such that $\underline{d}_{m_{n}}(B^{c})=0$ and $\lim_{j \in B}T_{j}x=0$ is dense in $X,$ then the set of $m_{n}$-distributionally near to zero vectors for $(T_{j})_{j\in {\mathbb N}}$
is residual in $X.$
\end{prop}

\begin{cor}\label{nemojaaa}
Suppose that $\lambda \in (0,1]$ and $(T_{j})_{j\in {\mathbb N}}$ is a sequence in $L(X,Y).$ If the set of those vectors $x\in X$ for which there exists a set $B\subseteq {\mathbb N}$ such that $\underline{d}_{1/\lambda}(B^{c})=0$
and $\lim_{j \in B}T_{j}x=0$ is dense in $X,$ then the set of $\lambda$-distributionally near to zero vectors for $(T_{j})_{j\in {\mathbb N}}$
is residual in $X.$
\end{cor}

Keeping in mind Proposition \ref{beton} and Proposition \ref{nemoj}, we can state the following extensions of the first parts in \cite[Theorem 3.7, Corollary 3.12]{mendoza} for $m_{n}$-distributional chaos:

\begin{prop}\label{aaa}
Let $(m_{n})\in {\mathrm R}$ and $(T_{j})_{j\in {\mathbb N}}$ be a sequence in $L(X,Y).$ 
Suppose that the set consisting of those vectors $x\in X$ for which there exists a set $B\subseteq {\mathbb N}$ such that $\underline{d}_{m_{n}}(B^{c})=0$ and $\lim_{j \in B}T_{j}x=0$ is dense in $X,$ as well as that there exist a number $\epsilon>0$, a zero sequence $(y_{k})$ in $X$ and a strictly increasing sequence $(N_{k})$ in ${\mathbb N}$ such that \eqref{prckof-frckof} holds.
Then the set of $m_{n}$-distributionally irregular vectors for $(T_{j})_{j\in {\mathbb N}}$ is residual in $X.$
\end{prop}

\begin{cor}\label{aaabar}
Let $\lambda \in (0,1]$ and $(T_{j})_{j\in {\mathbb N}}$ be a sequence in $L(X,Y).$ 
Suppose that the set consisting of those vectors $x\in X$ for which there exists a set $B\subseteq {\mathbb N}$ such that $\underline{d}_{1/\lambda}(B^{c})=0$ and $\lim_{j \in B}T_{j}x=0$ is dense in $X,$ as well as that there exist a number $\epsilon>0$, a zero sequence $(y_{k})$ in $X$ and a strictly increasing sequence $(N_{k})$ in ${\mathbb N}$ such that \eqref{prckofff-frckofff} holds. 
Then the set of $\lambda$-distributionally irregular vectors for $(T_{j})_{j\in {\mathbb N}}$ is residual in $X.$
\end{cor}

If $T\in L(X),$ then we are in a position to extend the assertion of \cite[Theorem 12]{2013JFA} for $m_{n}$-distributional chaos in a rather technical way. For these purposes, we introduce the $m_{n}$-Distributionally Chaotic Criterion and $\lambda$--Distributionally Chaotic Criterion in the following ways:
\begin{itemize}
\item[(DCC$_{m_{n}}$)] There exist a number $\epsilon>0$, a set $B\subseteq {\mathbb N},$ two sequences $(x_{k})$ and $(y_{k})$ in $X$ as well as a strictly increasing sequence $(N_{k})$ of natural numbers such that $\underline{d}_{m_{n}}(B^{c})=0$, $\lim_{n\in B}T^{n}x_{k}=0,$
$y_{k}\in \overline{span\{x_{n} : n\in {\mathbb N}\}},$ $\lim_{k\rightarrow \infty}y_{k}=0$ and \eqref{prckof-frckof} holds with $T_{j}\equiv T^{j};$
\end{itemize}
\begin{itemize}
\item[(DCC$_{\lambda}$)] There exist a number $\epsilon>0$, a set $B\subseteq {\mathbb N},$ two sequences $(x_{k})$ and $(y_{k})$ in $X$ as well as a strictly increasing sequence $(N_{k})$ of natural numbers such that $\underline{d}_{1/\lambda}(B^{c})=0$, $\lim_{n\in B}T^{n}x_{k}=0,$
$y_{k}\in \overline{span\{x_{n} : n\in {\mathbb N}\}},$ $\lim_{k\rightarrow \infty}y_{k}=0$ and \eqref{prckofff-frckofff} holds with $T_{j}\equiv T^{j}.$
\end{itemize}
Then we have:

\begin{thm}\label{ruza}
Suppose that $T\in L(X)$ and $(m_{n})\in {\mathrm R}.$ Then the following assertions are equivalent:
\begin{itemize}
\item[(i)] $T$ satisfies \emph{(DCC$_{m_{n}}$).}
\item[(ii)] There is an $m_{n}$-distributionally irregular vector for $T.$
\item[(iii)] $T$ is $m_{n}$-distributionally chaotic. 
\item[(iv)] There is an $m_{n}$-distributionally chaotic pair  of type $1$ for $T.$
\item[(v)] $T$ is $m_{n}$-distributionally chaotic of type $1.$
\item[(vi)] There is an $m_{n}$-distributionally irregular vector of type $1$ for $T.$
\end{itemize}
\end{thm} 

\begin{proof}
The equivalence of (i), (ii), (iii) and (iv) have been already proved. The implication (ii) $\Rightarrow$ (vi) is trivial, the implication (vi) $\Rightarrow$ (v) follows from the last statement in [A.], while the implication (v) $\Rightarrow$ (iv) follows directly from definition. This completes the proof.   
\end{proof}

\begin{cor}\label{ruza-prim}
Suppose that $T\in L(X)$ and $\lambda \in (0,1].$ Then the following assertions are equivalent:
\begin{itemize}
\item[(i)] $T$ satisfies \emph{(DCC$_{\lambda}$).}
\item[(ii)] There is a $\lambda$-distributionally irregular vector for $T.$
\item[(iii)] $T$ is $\lambda$-distributionally chaotic. 
\item[(iv)] There is a $\lambda$-distributionally chaotic pair for $T.$
\item[(v)] $T$ is $\lambda$-distributionally chaotic of type $1.$
\item[(vi)] There is a $\lambda$-distributionally irregular vector of type $1$ for $T.$
\end{itemize}
\end{cor} 

Concerning reiterative $m_{n}$-distributional chaos of types $1$ and $1+$, we will first state and prove the following two lemmas:

\begin{lem}\label{karakter}
Suppose that for each $j\in {\mathbb N},$ $T_{j} : D(T_{j}) \subseteq X \rightarrow Y$ is a linear operator, and $x\in \bigcap_{j\in {\mathbb N}}D(T_{j}).$
Then
there exists a finite number $r>0$ such that
$G_{x,0,m_{n}}(\delta)\leq c$
for $0<\delta<r$ iff there exist a finite number $d \in (0,\liminf_{n\rightarrow \infty}\frac{m_{n}}{n})$ and 
an infinite set $A\subseteq {\mathbb N}$ such that $\underline{d}_{m_{n}}(A^{c})=d$ and $\lim_{j\in A}T_{j}x=0.$  
\end{lem}

\begin{proof}
Suppose first that there exists a finite number $r>0$ such that
$G_{x,0,m_{n}}(\delta)\leq c$
for $0<\delta<r.$ Let $k\in {\mathbb N}$ and $k>1/r.$ Then there exists a positive integer $n_{k}\in {\mathbb N},$ ar large as we want to be, such that the segment $[1,m_{n_{k}}]$ contains at least $m_{n_{k}}-n_{k}(c+k^{-1})$ integers $j\in {\mathbb N}$ such that 
$d_{Y}(T_{j}x,0)<1/k.$ Let $A_{k}$ denote the collection of such numbers and let $A:=\bigcup_{k\in {\mathbb N},k>1/r}A_{k}.$ Then it can be easily seen that $\underline{d}_{m_{n}}(A^{c})=c$ and $\lim_{j\in A}T_{j}x=0.$ For the converse statement, we can simply prove that the existence of 
a finite number $d \in (0,\liminf_{n\rightarrow \infty}\frac{m_{n}}{n})$ and 
an infinite set $A\subseteq {\mathbb N}$ such that $\underline{d}_{m_{n}}(A^{c})=d$ and $\lim_{j\in A}T_{j}x=0$ implies 
$G_{x,0,m_{n}}(\delta)\leq c\equiv (d+\liminf_{n\rightarrow \infty}\frac{m_{n}}{n})/2$
for $0<\delta<r\equiv (\liminf_{n\rightarrow \infty}\frac{m_{n}}{n}-d)/2.$
\end{proof}

\begin{lem}\label{karaktera}
Suppose that for each $j\in {\mathbb N}$ we have $T_{j} \in L(X,Y)$. Denote by ${\mathcal X}$ the set consisting of all vectors $x\in X$ 
for which there exists an infinite set $A\subseteq {\mathbb N}$ such that $\underline{d}_{m_{n}}(A^{c})=d$ and $\lim_{j\in A}T_{j}x=0.$  Then ${\mathcal X}$ is residual if it is dense.
\end{lem}

\begin{proof}
For $k,\ m\in {\mathbb N} $, set
$$
 M_{k,m}:=\left\{x\in X : (\exists  n\in {\mathbb N})\,  m_{n}-\Bigl|\bigl\{1\leq j\leq m_{n} : p_{m}^{Y}(T_{j}x)<k^{-1}\bigr\}\Bigr| \geq n(c+k^{-1})\right\}.
 $$
Clearly, $M_{k,m}$ is open and dense (since $M_{k,m}\supseteq {\mathcal X}$), so the set
$X_1=\bigcap_{k,m} M_{k,m}$ is residual and contains ${\mathcal X}.$
\end{proof}

Keeping in mind the above lemmas, Remark \ref{pripazise} and the proof of \cite[Theorem 2.3]{bk}, we can deduce the following result:

\begin{thm}\label{reiteracxije}
Suppose $X$ is a Banach space and $T\in L(X).$ Then we have the following:
\begin{itemize}
\item[(i)] $T$ is reiteratively $m_{n}$-distributionally chaotic of type $1$ iff there exists a reiteratively $m_{n}$-distributionally irregular vector $x$ of type $1.$
\item[(ii)] $T$ is reiteratively $m_{n}$-distributionally chaotic of type $1+$ iff there exists a reiteratively $m_{n}$-distributionally irregular vector $x$ of type $1+.$
\end{itemize}
\end{thm} 

\begin{proof}
We will prove only (i) because the part (ii) can be deduced analogously. 
Due to
[B.], the existence of a reiteratively $m_{n}$-distributionally irregular vector $x$ of type $1$ implies that $T$ is reiteratively $m_{n}$-distributionally chaotic of type $1.$
On the other hand, Lemma \ref{karakter} and the consideration from Remark \ref{pripazise} together imply that a vector $x\in X$ is reiteratively $m_{n}$-distributionally irregular vector of type $1$ for $T$ iff $x$ is a
Li-Yorke irregular vector $x$ for $T$ and there exists an infinite set $A\subseteq {\mathbb N}$ such that $\underline{d}_{m_{n}}(A^{c})=d<\liminf_{n\rightarrow \infty}\frac{m_{n}}{n}$ and $\lim_{j\in A}T_{j}x=0. $ Suppose now that $T$ is reiteratively $m_{n}$-distributional chaotic of type $1$.
Then there exists a pair $(x,y)$ of distinct points such that $BF_{x,y,m_{n}}(\sigma) = 0$
for some $\sigma > 0$ and there exist $c \in (0,\liminf_{n\rightarrow \infty}\frac{m_{n}}{n})$ and $r > 0$ such that $G_{x,y,m_{n}}(\delta)\leq c$ for all $ 0 < \delta < r$.
Let $u=x-y$ and consider
$$
Y_1 = \overline{span}(Orb(u,T)),
$$
which is an infinite dimensional closed $T$-invariant subspace of $X$. Consider the operator $S\in B(Y_1)$ obtained by restricting $T$ to $Y_1$.

Then $S$ is reiteratively $m_{n}$-distributionally chaotic of type $1$ in $Y_1$ because $span\{u\}$ is a corresponding reiteratively $m_{n}$-distributionally chaotic scrambled set of type $1$.
Thus, by \cite[Corollary 5]{band}, $S$ has a residual set of points on $Y_1$ with orbit unbounded. Moreover, by Lemma \ref{karaktera}, $S$ has a residual set of points $z$ on $Y_1$
for which there exists an infinite set $A\subseteq {\mathbb N}$ such that $\underline{d}_{m_{n}}(A^{c})=d<\liminf_{n\rightarrow \infty}\frac{m_{n}}{n}$ and $\lim_{j\in A}T_{j}z=0.$
This yields that $S$ has a residual set of reiteratively $m_{n}$-distributionally irregular vectors of type $1,$ so that 
$T$ has a reiteratively $m_{n}$-distributionally irregular vector of type $1.$
\end{proof}

\section{Dense reiterative $m_{n}$-distributional chaos}\label{marek-MLOss}

In this section, we will see that the method proposed in the proof of \cite[Theorem 15]{2013JFA} provides a safe and sound way for the examination of dense reiterative $m_{n}$-distributional chaos of type $s$ in Fr\' echet spaces.
The first structural result of ours, which in combination with Theorem \ref{ruza} provides an extension of \cite[Theorem 15]{2013JFA} and the second part of \cite[Theorem 3.7]{mendoza}, reads as follows:

\begin{thm}\label{na-dobro}
Suppose that $X$ is separable, $(m_{n}) \in {\mathrm R},$ $(T_{j})_{j\in {\mathbb N}}$ is a sequence in $L(X,Y),$
$X_{0}$ is a dense linear subspace of $X,$ as well as:
\begin{itemize}
\item[(i)] $\lim_{j\rightarrow \infty}T_{j}x=0,$ $x\in X_{0},$
\item[(ii)] there exists an $m_{n}$-distributionally unbounded vector $y\in X$ for $(T_{j})_{j\in {\mathbb N}}.$
\end{itemize}
Then $(T_{j})_{j\in {\mathbb N}}$
is densely $m_{n}$-distributionally chaotic, and moreover, the scrambled set $S$ can be chosen to be a dense uniformly $m_{n}$-distributionally irregular submanifold of $X.$
\end{thm}

\begin{proof}
We will only outline the main details.
Without loss of generality, we may assume that 
\begin{align}\label{grozno}
p_{Y}^{m}(T_{j}x)\leq p_{j+m}(x)\ \ \mbox{for all }\ \ x\in X\ \ \mbox{ and }\ \ j,\ m\in {\mathbb N}
\end{align} 
as well as that a set $B\subseteq {\mathbb N}$ satisfies $\underline{d}_{m_{n}}(B^{c})=0$
and
$
\lim_{n\rightarrow \infty,n\in B}p^{1}_{Y}\bigl( T_{n}y \bigr)=+\infty.
$
Using the equality $\underline{d}_{m_{n}}(B^{c})=0$, we can construct a sequence $(x_{k})_{k\in {\mathbb N}}$ in $X_{0}$ and a strictly increasing sequence $(j_{k})_{k\in {\mathbb N}}$ of positive integers such that, for every $k\in {\mathbb N},$ one has: $p_{k}(x_{k})\leq 1,$ 
\begin{align*}
\Bigl| \bigl\{ 1  \leq j\leq m_{j_{k}} : p_{Y}^{1}(T_{j}x_{k})\geq k2^{k} \bigr\} \Bigr| \geq m_{j_{k}}-\frac{j_{k}}{k}
\end{align*}
and
\begin{align*}
\Bigl| \bigl\{1\leq j \leq m_{j_{k}} : p_{Y}^{k}(T_{j}x_{s})<1/k \bigr\}\Bigr|\geq m_{j_{k}}-\frac{j_{k}}{k},\quad s=1,\cdot \cdot \cdot,k-1.
\end{align*}
Take any strictly increasing sequence $(r_{q})_{q\in {\mathbb N}}$  in ${\mathbb N} \setminus \{1\}$
such that 
\begin{align}\label{miruga}
r_{q+1}\geq 1+r_{q}+m_{j_{r_{q}+1}}\mbox{ for all }q\in {\mathbb N}.
\end{align} 
Let $\alpha \in \{0,1\}^{\mathbb N}$ be a sequence defined by $\alpha_{n}=1$ iff $n=r_{q}$ for some $q\in {\mathbb N}.$
Further on, let $\beta \in \{0,1\}^{\mathbb N}$ 
contains an infinite number of $1'$s and let $\beta_{q}\leq \alpha_{q}$ for all $q\in {\mathbb N}.$ If $\beta_{r_{k}}=1$ for some $k\in {\mathbb N}$ and $x_{\beta}=\sum_{q=1}^{\infty}\beta_{r_{q}}x_{r_{q}}/2^{r_{q}},$ then for each $j\in [1,m_{j_{r_{k}}}]$ 
such that  $p_{Y}^{1}(T_{j}x_{k})\geq k2^{k}$
and $p_{Y}^{r_{k}}(T_{j}x_{s})<1/r_{k}$ for $s<r_{k},$
we have: $1+j\leq 1+m_{j_{r_{k}}} \leq 1+m_{j_{r_{q-1}}}\leq r_{q}$ for $q>k$): 
\begin{align*}
p_{Y}^{1}\bigl( T_{j}x_{\beta}\bigr) & \geq r_{k}-\sum_{q<k}\frac{p_{Y}^{1}(T_{j}x_{r_{q}})}{2^{r_{q}}} -\sum_{q>k}\frac{p_{Y}^{1}(T_{j}x_{r_{q}})}{2^{r_{q}}}
\\ & \geq r_{k}-\sum_{q<k}\frac{p_{Y}^{1}(T_{j}x_{r_{q}})}{2^{r_{q}}}-\sum_{q>k}\frac{p_{1+j}(x_{r_{q}})}{2^{r_{q}}}
\\ & \geq r_{k}-\sum_{q<k}\frac{1}{2^{r_{q}}r_{q}} -\sum_{q> k}\frac{1}{2^{r_{q}}} \geq r_{k}-1,
\end{align*}
which implies that for each fixed distinct numbers $\alpha_{1},\ \alpha_{2}\in {\mathbb K}$ there exists a positive integer $k_{0}(\alpha_{1},\alpha_{2})$ such that for each $k\geq k_{0}(\alpha_{1},\alpha_{2})$ one has:
\begin{align*}
d_{Y}\bigl(T_{k}\alpha_{1}x_{\beta},T_{k}\alpha_{1}x_{\beta}\bigr)\geq 2^{-1}\frac{|\alpha_{1}-\alpha_{2}|(r_{k}-1)}{1+|\alpha_{1}-\alpha_{2}|(r_{k}-1)}\geq 4^{-1};
\end{align*}
hence,
\begin{align}\label{lepo1}
\lim_{k\rightarrow \infty}\frac{\Bigl|\bigl\{ 1  \leq j\leq m_{j_{r_{k}}} : d_{Y}\bigl(T_{k}\alpha_{1}x_{\beta},T_{k}\alpha_{2}x_{\beta}\bigr)<4^{-1} \bigr\}\Bigr|}{j_{r_{k}}}=0.
\end{align}
Furthermore, if $j\in [1,m_{j_{r_{k}+1}}]$ and $p_{Y}^{r_{k}+1}(T_{j}x_{s})<1/(r_{k}+1)$ for $s<r_{k}+1,$
then we have $1+r_{k}+j\leq 1+r_{k}+m_{j_{r_{k}+1}}
\leq 1+r_{q-1}+m_{j_{r_{q-1}+1}}
\leq r_{q}$ due to \eqref{miruga} and therefore
\begin{align*}
p_{Y}^{r_{k}+1}\bigl( T_{j}x_{\beta}\bigr) & \leq \sum_{q\leq k}\frac{p_{Y}^{r_{k}+1}(T_{j}x_{r_{q}})}{2^{r_{q}}} +\sum_{q>k}\frac{p_{Y}^{r_{k}+1}(T_{j}x_{r_{q}})}{2^{r_{q}}}
\\ & \leq \sum_{q\leq k}\frac{1}{2^{r_{q}}(r_{k}+1)} +\sum_{q>k}\frac{p_{1+j+r_{k}}(x_{r_{q}})}{2^{r_{q}}}
\\ & \leq \frac{1}{2(r_{k}+1)}+\sum_{q>k}\frac{1}{2^{r_{q}}} \leq \frac{1}{r_{k}+1},
\end{align*}
which clearly implies
\begin{align}\label{metrisa}
d_{Y}\bigl( T_{j}x_{\beta},0\bigr)=
\sum
\limits_{q=1}^{r_{k}+1}\frac{1}{2^{q}}\frac{p_{q}(T_{j}x_{\beta})}{1+p_{q}(T_{j}x_{\beta})}+\sum
\limits_{q=r_{k}+1}^{\infty}\frac{1}{2^{q}}\frac{p_{q}(T_{j}x_{\beta})}{1+p_{q}(T_{j}x_{\beta})}\leq \frac{1}{r_{k}+2}+\frac{1}{2^{r_{k}}}.
\end{align}
This yields that for each $\epsilon>0$ and for each pair of distinct numbers $\alpha_{1},\ \alpha_{2}\in {\mathbb K}$ we have:
\begin{align}\label{lepo2}
\lim_{k\rightarrow \infty}\frac{\Bigl|\bigl\{ 1  \leq j\leq m_{j_{r_{k}}} : d_{Y}\bigl(T_{k}\alpha_{1}x_{\beta},T_{k}\alpha_{2}x_{\beta}\bigr)\geq \epsilon \bigr\}\Bigr|}{j_{r_{k}}}=0.
\end{align}
By \eqref{lepo1}-\eqref{lepo2}, we get that the sequence $(T_{j})_{j\in {\mathbb N}}$ is $m_{n}$-distributionally chaotic, with $S=span\{x_{\beta}\}$ as a $1/4$-scrambled set. The final statement of theorem now follows similarly as in the proof of 
\cite[Theorem 15]{2013JFA}.
\end{proof}

The following corollaries are immediate:

\begin{cor}\label{na-dobrorade}
Suppose that $X$ is separable, $\lambda \in (0,1],$ $(T_{j})_{j\in {\mathbb N}}$ is a sequence in $L(X,Y),$
$X_{0}$ is a dense linear subspace of $X,$ as well as:
\begin{itemize}
\item[(i)] $\lim_{j\rightarrow \infty}T_{j}x=0,$ $x\in X_{0},$
\item[(ii)] there exists a $\lambda$-distributionally unbounded vector $y\in X$ for $(T_{j})_{j\in {\mathbb N}}.$
\end{itemize}
Then $(T_{j})_{j\in {\mathbb N}}$
is densely $\lambda$-distributionally chaotic, and moreover, the scrambled set $S$ can be chosen to be a dense uniformly $\lambda$-distributionally irregular submanifold of $X.$
\end{cor}

\begin{cor}\label{na-dobrwo}
Suppose that $X$ is separable, $(T_{j})_{j\in {\mathbb N}}$ is a sequence in $L(X,Y),$
$X_{0}$ is a dense linear subspace of $X,$ as well as:
\begin{itemize}
\item[(i)] $\lim_{j\rightarrow \infty}T_{j}x=0,$ $x\in X_{0},$
\item[(ii)] there exist $y\in X$ and $m\in {\mathbb N}$ such that $\lim_{j \rightarrow \infty}p_{m}^{Y}(T_{j}y)=+\infty.$
\end{itemize}
Then $(T_{j})_{j\in {\mathbb N}}$
is densely $m_{n}$-distributionally chaotic for each sequence $(m_{n}) \in {\mathrm R},$ and moreover, the corresponding scrambled set $S$ can be chosen to be a dense uniformly $m_{n}$-distributionally irregular submanifold of $X.$ 
\end{cor}

\begin{rem}\label{peris}
It is worth noting that Corollary \ref{na-dobrorade} provides a generalization of \cite[Theorem 16]{2013JFA} for sequences of operators, where the case $\lambda=1$ has been considered. On the other hand, Corollary \ref{na-dobrwo} provides a generalization of \cite[Corollary 17]{2013JFA} for sequences of operators; in this statement, the authors have shown that the Godefroy-Schapiro criterion (see \cite{godefroy} and \cite[Theorem 3.1]{erdper}) implies dense distributional chaos. 
\end{rem}

Suppose now that $T_{j} : D(T_{j})\subseteq X \rightarrow X$ is a linear mapping, $C\in L(X)$ is an injective
mapping with dense range, as well as
\begin{equation}\label{cea}
R(C)\subseteq D(T_{j})\mbox{ and }T_{j}C\in L(X)\mbox{ for all }j\in {\mathbb N}.
\end{equation}
Then (\ref{cea}) implies that,
for every $j\in {\mathbb N},$ the mapping $T_{j}' : R(C) \rightarrow X$ defined by
$T_{j}'(Cx):=T_{j}Cx,$ $x\in X,$ $j\in {\mathbb N}$ is an element of the space $L([R(C)],X)$. By Theorem \ref{na-dobro},
we immediately obtain the following corollary.

\begin{cor}\label{cea1}
Let $(m_{n}) \in {\mathrm R},$ let $\lambda \in (0,1],$
and let the above conditions hold.
\begin{itemize}
\item[(i)] Suppose that $X$ is separable, $X_{0}$ is a dense linear subspace of
$X,$ as well as:
\begin{itemize}
\item[(a)] $\lim_{j\rightarrow \infty}T_{j}Cx=0,$ $x\in X_{0},$
%\item[(b)] there exists
%$q\in {\mathbb N}$ such that
%\begin{equation}\label{DenseChaoss}
%p_{m}\bigl(T^{k}Cx\bigr)\leq p_{k+m+q}(x),\quad x\in X,\ k,\ m \in {\mathbb N},
%\end{equation}
\item[(b)] there exist $x\in X,$ $m\in {\mathbb N}$ and a set $B\subseteq {\mathbb N}$ such that
$\underline{d}_{m_{n}}(B^{c})=0$ and $\lim_{j\rightarrow \infty ,j\in B}p_{m}(T_{j}Cx)=\infty,$ resp.
$\lim_{j\rightarrow \infty ,j\in B}\|T_{j}Cx\|=\infty$ if $X$ is a Banach space.
\end{itemize}
Then the sequence $(T_{j})_{j\in {\mathbb N}}$
is densely $m_{n}$-distributionally chaotic for each sequence $(m_{n}) \in {\mathrm R},$ and moreover, the scrambled set $S$ can be chosen to be a dense uniformly $m_{n}$-distributionally irregular submanifold of $X.$
\item[(ii)]
Suppose that $X$ is separable, $X_{0}$ is a dense linear subspace of
$X,$ as well as:
\begin{itemize}
\item[(a)] $\lim_{j\rightarrow \infty}T_{j}Cx=0,$ $x\in X_{0},$
%\item[(b)] there exists
%$q\in {\mathbb N}$ such that
%\begin{equation}\label{DenseChaoss}
%p_{m}\bigl(T^{k}Cx\bigr)\leq p_{k+m+q}(x),\quad x\in X,\ k,\ m \in {\mathbb N},
%\end{equation}
\item[(b)] there exist $x\in X,$ $m\in {\mathbb N}$ and a set $B\subseteq {\mathbb N}$ such that
$\underline{d}_{1/\lambda}(B^{c})=0$ and $\lim_{j\rightarrow \infty ,j\in B}p_{m}(T_{j}Cx)=\infty,$ resp.
$\lim_{j\rightarrow \infty ,j\in B}\|T_{j}Cx\|=\infty$ if $X$ is a Banach space.
\end{itemize}
Then the sequence $(T_{j})_{j\in {\mathbb N}}$
is densely $\lambda$-distributionally chaotic, and moreover, the scrambled set $S$ can be chosen to be a dense uniformly $\lambda$-distributionally irregular submanifold of $X.$
\item[(iii)] Suppose that $X$ is separable, $X_{0}$ is a dense linear subspace of
$X,$ as well as:
\begin{itemize}
\item[(a)] $\lim_{j\rightarrow \infty}T_{j}Cx=0,$ $x\in X_{0},$
%\item[(b)] there exists
%$q\in {\mathbb N}$ such that
%\begin{equation}\label{DenseChaoss}
%p_{m}\bigl(T^{k}Cx\bigr)\leq p_{k+m+q}(x),\quad x\in X,\ k,\ m \in {\mathbb N},
%\end{equation}
\item[(b)] there exist $x\in X$ and $m\in {\mathbb N}$ such that
$\lim_{j\rightarrow \infty }p_{m}(T_{j}Cx)=+\infty$.
\end{itemize}
Then the sequence $(T_{j})_{j\in {\mathbb N}}$
is densely $m_{n}'$-distributionally chaotic for each sequence $(m_{n}') \in {\mathrm R},$ and moreover, the corresponding scrambled set $S$ can be chosen to be a uniformly $m_{n}'$-distributionally irregular submanifold of $X.$
\end{itemize}
\end{cor}

In \cite[Example 3.8-Example 3.10, Corollary 3.12]{mendoza}, we have considered only orbits of linear operators; it is clear that Corollary \ref{cea1}(ii) provides an extension of the second part in \cite[Corollary 3.12]{mendoza}, where the case $\lambda=1$ has been analyzed, as well as that Corollary \ref{cea1}(iii) can be used to further improve the conclusions obtained in \cite[Example 3.8-Example 3.10]{mendoza} for certain classes of unbounded differential operators in Banach spaces:

\begin{example}\label{idiote}
(see \cite[Example 3.9]{mendoza} and references cited therein) Denote by ${\mathcal F}$ and ${\mathcal F}^{-1}$ the Fourier transform on the real line and its inverse transform, respectively. 
Assume that $X:=L^{2}({\mathbb R}),$ $c>b/2>0,$ $\Omega:=\{ \lambda \in
{\mathbb C} : \Re \lambda<c-b/2\}$ and ${\mathcal A}_{c}u:=u^{\prime
\prime}+2bxu^{\prime}+cu$ is the bounded perturbation of the
one-dimensional Ornstein-Uhlenbeck operator acting with domain
$D({\mathcal A}_{c}):=\{u\in L^{2}({\mathbb R})\ \cap \
W^{2,2}_{loc}({\mathbb R}) : {\mathcal A}_{c}u\in L^{2}({\mathbb
R})\}.$ 
Then it is well known that ${\mathcal A}_{c}$ generates a strongly continuous semigroup,
$\Omega \subseteq \sigma_{p}({\mathcal A}_{c}),$ and for any open
connected subset $\Omega'$ of $\Omega$ which admits a cluster point
in $\Omega,$ one has $E=\overline{span\{g_{i}(\lambda) : \lambda \in
\Omega',\ i=1,2 \}},$ where $g_{1} : \Omega \rightarrow X$ and
$g_{2} :  \Omega \rightarrow X$ are defined by
$g_{1}(\lambda):={\mathcal F}^{-1}(e^{-\frac{\xi^{2}}{2b}}\xi
|\xi|^{-(2+\frac{\lambda-c}{b})})(\cdot),$ $\lambda \in \Omega$ and
$g_{2}(\lambda):={\mathcal F}^{-1}(e^{-\frac{\xi^{2}}{2b}}
|\xi|^{-(1+\frac{\lambda-c}{b})})(\cdot),$ $\lambda \in \Omega.$ Assume that $(P_{j}(z))_{j\in {\mathbb N}}$ is a sequence of non-zero complex polynomials such that there exists an open connected subset $\Omega'$ of $\Omega$ such that $\lim_{j\rightarrow \infty}P_{j}(\lambda)=0,$ $\lambda \in \Omega'$ as well as that there exists a number $\lambda \in \Omega$ such that $|P_{j}(\lambda)|>1.$
Due to Corollary \ref{cea1}(ii), we get that the sequence of operators $(P_{j}({\mathcal A}_{c}))_{j\in {\mathbb N}}$ is densely $m_{n}$-distributionally chaotic for each sequence $(m_{n}) \in {\mathrm R}.$ 
\end{example}

Assume, for the time being, that $X$ is a  Fr\' echet sequence space in which $(e_{n})_{n\in {\mathbb N}}$ is a basis and $(\omega_{n})_{n\in {\mathbb N}}$ is a sequence of positive weights; for more details, see \cite[Section 4.1]{erdper}. Consider
the unilateral weighted backward shift $T_{\omega} : D(T_{\omega}) \subseteq X \rightarrow X,$
given by 
\begin{align}\label{oladilo}
T_{\omega}& \bigl\langle x_{n}\bigr\rangle_{n\in {\mathbb N}}:=\bigl\langle w_{n}x_{n+1}\bigr\rangle_{n\in {\mathbb N}},\quad \bigl\langle x_{n}\bigr\rangle_{n\in {\mathbb N}}\in X,
\end{align}
about which we assume that it is not necessarily continuous. Albeit it is without scope of this paper to consider $m_{n}$-distributionally chaotic properties of unbounded bilateral weighted shift operators (see \cite{ddc}, where we have recently initiated the study of disjoint distributionally chaotic properties of  such operators), we will state here only one result regarding this question, which can be deduced with the help of Corollary \ref{cea1}(iii) and the proof of \cite[Theorem 4.11]{ddc}:

\begin{thm}\label{rikardinjo}
Let $X:=l^{p}$ for some $1\leq p<\infty$ or $X:=c_{0},$ and let $(m_{n}) \in {\mathrm R}.$
Suppose that there exists a bounded sequence $(a_{n})_{n\in {\mathbb N}}$ of positive reals such that for each $k\in {\mathbb N}$ we have
$$
B_{k}:=\sup_{n\in {\mathbb N}}\Biggl[a_{k+n}\prod_{i=n}^{k+n-1}\omega_{i}\Biggr]<\infty.
$$ 
Then the following holds:
\begin{itemize}
\item[(i)] If $p\neq 2$ and $X=l^{p}$ or $X=c_{0},$ as well as $\sum \limits_{k=1}^{\infty}\frac{1}{B_{k}}<\infty ,$
\end{itemize}
or
\begin{itemize}
\item[(ii)] $X=l^{2}$ and $\sum \limits_{k=1}^{\infty}\frac{1}{B_{k}^{2}}<\infty ,$
\end{itemize}
then the operator $T_{\omega}$ is densely $m_{n}$-distributionally chaotic.
\end{thm}

An illustrative example of application can be simply given:

\begin{example}\label{tekma}
Let $X:=l^{p}$ for some $p\in [1,\infty)$ or $X:=c_{0},$ and
let $\omega_{n}:=n^{j}$ for some $j>0$. Then we can apply Theorem \ref{rikardinjo} in order to see that the operator $T_{\omega}$ is densely $m_{n}$-distributionally chaotic for any $(m_{n})\in {\mathrm R}.$
\end{example}

In the following example, we construct $m_{n}$-distributionally unbounded vectors directly (cf. also Example \ref{rdc} below):

\begin{example}\label{tekma-prim}
Let us recall that L. Luo and B. Hou has considered the case $X:=l^{1}({\mathbb N})$ and $\langle \omega_{n}\rangle_{n\in {\mathbb N}}:= \langle \frac{2n}{2n-1}\rangle_{n\in {\mathbb N}},$ showing that the corresponding 
operator $T_{\omega}$ is topologically mixing, absolutely Ces\`aro bounded and therefore not distributionally chaotic (\cite{turkish-notes}). The analysis has been recently continued in \cite{2018JMMA}, where it has been shown that the operator $T_{\omega}$ cannot be distributionally chaotic of type $3.$ Consider now the following cases:
\begin{itemize}
\item[1.] $X:=l^{p}({\mathbb N})$ for some $p\in (1,\infty).$
Due to Stirling's formula, we have that $\beta (n):=\prod_{j=1}^{n}\omega_{i} \sim \sqrt{\pi n},$ $n\rightarrow +\infty.$ 
Applying e.g. \cite[Proposition 3.1]{turkish-notes}, we get that the operator $T_{\omega}$ is topologically mixing iff $p>1$ as well as that  $T_{\omega}$ is chaotic iff $p>2.$
Now we will prove that for each $p>1$ and each sequence $(m_{n})\in {\mathrm R},$ the operator $T_{\omega}$ is densely $m_{n}$-distributionally chaotic. 
Let $\epsilon \in (0,1)$ be arbitrarily chosen. Then it is clear that $\langle x_{n} \rangle_{n\in {\mathbb N}}:=\langle n^{-(1+\epsilon)/p} \rangle_{n\in {\mathbb N}} \in X$ as well that there exists a positive finite constant $c>0$ such that 
\begin{align}\label{zelje-eho}
\omega_{n}\omega_{n+1}\cdot \cdot \cdot \omega_{n+j}=\beta (n+j)/\beta (n-1) \geq c \sqrt{1+\frac{j}{n}} \ \ \mbox{ for all }n,\ j \in{\mathbb N}.
\end{align}
We will prove that the vector $x=\langle x_{n} \rangle_{n\in {\mathbb N}}$ satisfies $\lim_{j\rightarrow \infty}\|T_{\omega}^{j}x\|=+\infty,$ so that the final conclusion follows by applying Theorem \ref{djubre}.
Using \eqref{zelje-eho} and the well known result regarding the estimates of  partial sums defining the Riemman zeta function, we get that
\begin{align*}
\bigl\|T_{\omega}^{j}x\bigr\|&=\Biggl( \sum_{n=1}^{\infty}\bigl|\omega_{n}\omega_{n+1}\cdot \cdot \cdot \omega_{n+j}x_{n+j+1}\bigr|^{p}\Biggr)^{1/p}
\\ & \geq c \Biggl( \sum_{n=1}^{\infty}\Bigl(1+\frac{j}{n}\Bigr)^{p/2}\bigl|x_{n+j+1}\bigr|^{p}\Biggr)^{1/p}\geq c\sqrt{j}\Biggl( \sum_{n=1}^{\infty}n^{-\frac{p}{2}}\bigl|x_{n+j+1}\bigr|^{p}\Biggr)^{1/p}
\\ & \geq c\sqrt{j}\Biggl( \sum_{n=1}^{\infty}\frac{1}{(n+1+j)^{\frac{p}{2}+\frac{1+\epsilon}{2}}}\Biggr)^{1/p}
\\ & \geq c\sqrt{j}j^{\frac{1-\bigl[\frac{p}{2}+\frac{1+\epsilon}{2}\bigr]}{p}}
=cj^{\frac{1-\epsilon}{2p}} \rightarrow +\infty,\quad j\rightarrow+\infty.
\end{align*}
\item[2.] $X:=c_{0}({\mathbb N})$. Then $\omega_{n}\geq1$ for all $n\in {\mathbb N}$ and the vector $x:=\langle 1/n \rangle_{n\in {\mathbb N}}$ satisfies $\lim_{j\rightarrow \infty}\|T_{\omega}^{j}x\|=+\infty,$ which
implies by Theorem \ref{djubre} that the operator $T_{\omega}$ is densely $m_{n}$-distributionally chaotic for any sequence $(m_{n}) \in {\mathbb R}.$ Due to \cite[Theorem 4.8]{erdper}, the operator $T_{\omega}$ is both topologically mixing and chaotic.
\end{itemize}
\end{example}

We are turning back to the case in which $X$ is a general Fr\' echet space under our consideration, by
stating the following result:

\begin{thm}\label{na-dobrorc} 
Suppose that $X$ is separable, $(T_{j})_{j\in {\mathbb N}}$ is a sequence in $L(X,Y),$
$X_{0}$ is a dense linear subspace of $X,$ as well as:
\begin{itemize}
\item[(i)] $\lim_{j\rightarrow \infty}T_{j}x=0,$ $x\in X_{0},$
%\item[(b)] there exists
%$q\in {\mathbb N}$ such that
%\begin{equation}\label{DenseChaos}
%p_{m}^{Y}\bigl(T_{k}x\bigr)\leq p_{k+m+q}(x),\quad x\in X,\ k,\ m \in {\mathbb N},
%\end{equation}
\item[(ii)] there exists a reiteratively $m_{n}$-distributionally unbounded vector $x\in X$ for $(T_{j})_{j\in {\mathbb N}}.$
\end{itemize}
Then there exists a dense uniformly reiteratively $m_{n}$-distributionally irregular manifold of type $1^+$ for $(T_{j})_{j\in {\mathbb N}},$
and particularly, $(T_{j})_{j\in {\mathbb N}}$
is densely reiteratively $m_{n}$-distributionally chaotic of type $1^+$.
\end{thm}

\begin{proof}
We will only outline the main details.
Without loss of generality, we may assume that \eqref{grozno} holds
as well as that, due to (b), 
there exists a set $B\subseteq {\mathbb N}$ such that $\underline{Bd}_{l;m_{n}}(B^{c})=0$ and $\lim_{j\in B}p_{Y}^{1}(T_{j}y)=+\infty.$ 
Using this,  we can construct a sequence $(x_{k})_{k\in {\mathbb N}}$ in $X_{0},$ a strictly increasing sequence $(j_{k})_{k\in {\mathbb N}}$ of positive integers and 
a sequence $(B_{k})_{k\in {\mathbb N}}$ of subsets consisted of positive integers such that $\sup B_{k} <\inf B_{k+1}$ for all $k\in {\mathbb N},$
as well as that, for every $j,\ k\in {\mathbb N},$ one has: 
$B_{k}\subseteq [j_{k}+1,j_{k}+m_{s_{k}}] \cap B,$ $|B_{k}|\geq m_{s_{k}}-(s_{k}/k),$
$p_{k}(x_{k})\leq 1,$ 
$p_{Y}^{1}(T_{j}x_{k})\geq k2^{k}$ for $j\in B_{k},$ $p_{Y}^{k}(T_{j}x_{s})\leq 1/(k+1)$ for all $s\in {\mathbb N}_{k-1}
$ and $j\in {\mathbb N}$ with 
$j\in B_{k},$ as well as: 
$$
\frac{\bigl| \{1\leq j \leq m_{j_{k}} : p_{Y}^{k}(T_{j}x_{s})<1/k \}\bigr|}{j_{k}}\geq 1-k^{-2},\quad s\in {\mathbb N}_{k-1}.
$$
In particular, for each $k\in {\mathbb N}$ we have:
\begin{align}\label{zajebano}
p_{Y}^{1}(T_{j}x_{s})\leq 1/(k+1),\quad s\in {\mathbb N}_{k-1},\ 
j\in B_{k}.
\end{align}
Set $B':=\bigcup_{k\in {\mathbb N}}B_{k}.$ Since $\underline{Bd}_{l;m_{n}}(B^{c})=0,$ it is clear that\\ $\liminf_{s\rightarrow \infty}\inf_{n\in {\mathbb N}}\frac{|B^{c} \cap [n+1,n+m_{s}]|}{s}=0,$ which simply implies that\\ $\liminf_{s\rightarrow \infty}\inf_{n\in {\mathbb N}}\frac{|(B')^{c} \cap [n+1,n+m_{s}]|}{s}=0.$ On the other hand, Lemma \ref{mile-duo}(i) implies that $\underline{Bd}_{l;m_{n}}((B')^{c})=0,$
provided that $(B')^{c}$ is finite of infinite non-syndetic. If  $(B')^{c}$ is syndetic, then it is clear that there exists a finite constant $d>0$ such that $\underline{Bd}_{l;m_{n}}((B')^{c})\leq \liminf_{s\rightarrow \infty}(m_{s}/s)\leq d\liminf_{s\rightarrow \infty}\inf_{n\in {\mathbb N}}\frac{|(B')^{c} \cap [n+1,n+m_{s}]|}{s}=0;$ summa summarum, we have $\underline{Bd}_{l;m_{n}}((B')^{c})=0.$
Take now any strictly increasing sequence $(r_{q})_{q\in {\mathbb N}}$ of positive integers 
such that $r_{q+1}\geq 1+r_{q}+j_{r_{q}+1}$ for all $q\in {\mathbb N}.$ 
Let $\alpha \in \{0,1\}^{\mathbb N}$ be a sequence defined by $\alpha_{n}=1$ iff $n=r_{q}$ for some $q\in {\mathbb N}.$
Further on, let $\beta \in \{0,1\}^{\mathbb N}$ 
contains an infinite number of $1'$s and let $\beta_{q}\leq \alpha_{q}$ for all $q\in {\mathbb N}.$ If $\beta_{r_{k}}=1$ for some $k\in {\mathbb N}$ and $x_{\beta}=\sum_{q=1}^{\infty}\beta_{r_{q}}x_{r_{q}}/2^{r_{q}},$ then for each $j\in B_{j_{r_{k}}}$ we have (see \eqref{zajebano} and observe that for such values of $j$ we have $1+q\leq 1+j_{r_{k}}\leq r_{k+1}\leq r_{q}$ for $q>k$): 
\begin{align*}
p_{Y}^{1}\bigl( T_{j}x_{\beta}\bigr) & \geq r_{k}-\sum_{q<k}\frac{p_{Y}^{1}(T_{j}x_{r_{q}})}{2^{r_{q}}} -\sum_{q>k}\frac{p_{Y}^{1}(T_{j}x_{r_{q}})}{2^{r_{q}}}
\\ & \geq r_{k}-\sum_{q<k}\frac{p_{Y}^{1}(T_{j}x_{r_{q}})}{2^{r_{q}}}-\sum_{q>k}\frac{p_{1+j}(x_{r_{q}})}{2^{r_{q}}}
\\ & \geq r_{k}-\sum_{q<k}\frac{1}{2^{r_{q}}(1+r_{q})} -\sum_{q> k}\frac{1}{2^{r_{q}}} \geq r_{k}-2.
\end{align*}
Hence, $\lim_{j\in B'}p_{Y}^{1}(T_{j}x_{\beta})=
\infty$ and the vector $x_{\beta}$ is reiteratively $m_{n}$-distributionally $1$-unbounded. Arguing in the same way
as in the proof of Theorem \ref{na-dobro}, we get that $x_{\beta}$ is $m_{n}$-distributionally near to zero. The conclusion of theorem
now simply follows by copying the final part of proof of \cite[Theorem 15]{2013JFA}.
\end{proof}

For the orbits of a linear continuous operator $T\in L(X),$ the condition (ii) in the formulation of Theorem \ref{na-dobro} is equivalent with its unboundedness; therefore, Theorem \ref{na-dobro} provides an extension of \cite[Corollary 21]{band}. We can also state the following corollary:

\begin{cor}\label{rdctype}
Suppose $T\in L(X),$ $X$ is a separable Banach space,
$X_{0}$ is a dense linear subspace of $X$ and
$\lim_{j\rightarrow \infty}T^{j}x=0,$ $x\in X_{0}.$
Then $T$ is densely Li-Yorke chaotic iff $T$ is densely reiteratively $m_{n}$-distributionally chaotic of type $0$ or $1^+.$ 
\end{cor}

\begin{proof}
All non-trivial that we need to show is that the dense Li-Yorke chaos for $T$ implies dense reiterative $m_{n}$-distributional chaos of type $1^+$ for $T.$ Assume that $T$ has an unbounded orbit. Arguing as in the first part of proof of Proposition \ref{ekvivalentno}, we get that $T$ admits a reiteratively $m_{n}$-distributionally unbounded orbit. Applying Theorem \ref{na-dobro}, we get that there exists a dense uniformly reiteratively $m_{n}$-distributionally irregular manifold of type $1^+$ for $T.$
\end{proof}

\begin{example}\label{molim-te}
It is worth noting that Corollary \ref{rdctype}
can be applied in the analysis of multiplication operators and their adjoints in Hilbert spaces. 
For example, the Li-Yorke chaos of an adjoint multiplication operator $M_{\varphi}^{\ast}$ considered in \cite[Theorem 26(ii)]{band} is not only equivalent with its hypercyclicity but also with reiterative distributional chaos of type $s$ for any $s\in \{0,1,1^+\};$ this follows from Corollary \ref{rdctype} and the arguments contained in the proof of \cite[Theorem 4.5]{godefroy}.
\end{example}

Arguing as in the proof of \cite[Theorem 15]{2013JFA} and Theorem \ref{na-dobro} above, we can similarly deduce the following extension of last mentioned result for general sequences of linear continuous operators (the extension is proper even for orbits of linear continuous operators):

\begin{thm}\label{na-dobro1}
Suppose that $X$ is separable, $(m_{n}) \in {\mathbb R},$ $(T_{j})_{j\in {\mathbb N}}$ is a sequence in $L(X,Y),$
$X_{0}$ is a dense linear subspace of $X,$ as well as:
\begin{itemize}
\item[(i)] $\lim_{j\rightarrow \infty}T_{j}x=0,$ $x\in X_{0},$
%\item[(b)] there exists
%$q\in {\mathbb N}$ such that
%\begin{equation}\label{DenseChaos}
%p_{m}^{Y}\bigl(T_{k}x\bigr)\leq p_{k+m+q}(x),\quad x\in X,\ k,\ m \in {\mathbb N},
%\end{equation}
\item[(ii)] there exists a vector $y\in X$ such that the sequence $(T_{j}y)_{j\in {\mathbb N}}$ is unbounded.
\end{itemize}
Then there exist a number $m\in {\mathbb N}$ and a dense
submanifold $W$ of $X$ consisting of those vectors $x$ for which 
the sequence $(p_{m}(T_{j}x))_{j\in {\mathbb N}}$ is unbounded and which are $m_{n}$-distributionally near to zero for $(T_{j})_{j\in {\mathbb N}}.$
In particular, $(T_{j})_{j\in {\mathbb N}}$
is densely Li-Yorke chaotic.
\end{thm}

\begin{rem}\label{jklp}
\begin{itemize}
\item[(i)]
It is worth noting that Theorem \ref{na-dobro} and Theorem \ref{na-dobro1} provide extensions of \cite[Theorem 3.1, Corollary 3.2]{bk}, where the orbits of an operator in Banach space have been considered.
\item[(ii)] A slight generalization of Theorem \ref{na-dobro1} for disjoint Li-Yorke chaotic operators has been recently established and proved in \cite{ddc-ly}.
\end{itemize}
\end{rem}

Now we will state and prove the following result, which is closely linked with Theorem \ref{na-dobro} and Theorem \ref{na-dobro1}:

\begin{thm}\label{na-dobro2}
Suppose that $X$ is separable, $(T_{j})_{j\in {\mathbb N}}$ is a sequence in $L(X,Y),$
$X_{0}$ is a dense linear subspace of $X,$ as well as:
\begin{itemize}
\item[(i)] $\lim_{j\rightarrow \infty}T_{j}x=0,$ $x\in X_{0},$
%\item[(b)] there exists
%$q\in {\mathbb N}$ such that
%\begin{equation}\label{DenseChaos}
%p_{m}^{Y}\bigl(T_{k}x\bigr)\leq p_{k+m+q}(x),\quad x\in X,\ k,\ m \in {\mathbb N},
%\end{equation}
\item[(ii)] there exist $m\in {\mathbb N},$ $c>0,$ $x\in X$ and set $B\subseteq {\mathbb N}$ such that $\overline{d}_{m_{n}}(B)=c$ and $\lim_{j\in B}p_{m}^{Y}(T_{j}x)=+\infty.$ 
\end{itemize}
Then there exists a dense $m_{n}$-distributionally irregular manifold of type $2$ for $(T_{j})_{j\in {\mathbb N}},$
and particularly, $(T_{j})_{j\in {\mathbb N}}$
is densely $m_{n}$-distributionally chaotic of type $2.$
\end{thm}

\begin{proof}
In order to see that the equation \eqref{grozno} can be again used with $m=1,$ it suffices to observe the following:
\begin{itemize}
\item[1.] Suppose that $Y$ is a pure Fr\' echet space.
Then we can always construct a fundamental system
$(p_{n}'(\cdot))_{n\in {\mathbb N}}$ of
increasing seminorms on the space $X,$ inducing the same topology on $X$,
so that $p_{m}^{Y}(T_{j}x)\leq p_{j+m}'(x),$ $x\in X,$ $j,\ m\in {\mathbb N}.$ 
\item[2.] If $(Y,\| \cdot\|_{Y})$ is a Banach space, then we can renorm $Y$ by using the fundamental system $(p_{m}^{Y}\equiv m\|\cdot\|)_{m\in {\mathbb N}},$ turning $Y$ into a linearly and topologically homeomorphic Fr\' echet space for which the induced metric $d_{Y}: Y\times Y \rightarrow [0,1]$ satisfies  $d_{y}(x,y)\leq \|x-y\|\sum_{m=0}^{\infty}m2^{-m}$ for all $x,\ y\in Y$ as well as the implication: $\|x\|>k^{-1}$ for some $x\in X$ and $k\in {\mathbb N}\Rightarrow d_{Y}(x,0)\geq 1/2(k+1).$
\end{itemize}
In our new situation,
we can construct a sequence $(x_{k})_{k\in {\mathbb N}}$ in $X_{0}$ and a strictly increasing sequence $(j_{k})_{k\in {\mathbb N}}$ of positive integers such that, for every $k\in {\mathbb N},$ one has: $p_{k}(x_{k})\leq 1,$ 
$$
\Bigl| \{1\leq j \leq m_{j_{k}} : p_{Y}^{1}(T_{j}x_{k})>k2^{k} \}\Bigr|\geq cj_{k}\Bigl(1-k^{-2}\Bigr)
$$
and
$$
\Bigl| \{1\leq j \leq m_{j_{k}} : p_{Y}^{k}(T_{j}x_{s})<1/k \}\Bigr|\geq m_{j_{k}}-\frac{c}{2}\frac{j_{k}}{k},\quad s\in {\mathbb N}_{k-1}.
$$
The remaining part of proof can be deduced by repeating verbatim the arguments used in the proofs of Theorem \ref{na-dobro} and \cite[Theorem 15]{2013JFA}. 
\end{proof}

We can also clarify the following statement regarding the existence of dense $m_{n}$-distributionally irregular manifolds of type $2_{Bd}:$

\begin{thm}\label{na-dobro3} 
Suppose that $X$ is separable, $(T_{j})_{j\in {\mathbb N}}$ is a sequence in $L(X,Y),$
$X_{0}$ is a dense linear subspace of $X,$ as well as:
\begin{itemize}
\item[(i)] $\lim_{j\rightarrow \infty}T_{j}x=0,$ $x\in X_{0},$
%\item[(b)] there exists
%$q\in {\mathbb N}$ such that
%\begin{equation}\label{DenseChaos}
%p_{m}^{Y}\bigl(T_{k}x\bigr)\leq p_{k+m+q}(x),\quad x\in X,\ k,\ m \in {\mathbb N},
%\end{equation}
\item[(ii)] there exist $m\in {\mathbb N},$ $c>0,$ $x\in X$ and set $B\subseteq {\mathbb N}$ such that $\overline{Bd}_{l:m_{n}}(B)=c$ and $\lim_{j\in B}p_{m}^{Y}(T_{j}x)=+\infty.$
\end{itemize}
Then there exists a dense $m_{n}$-distributionally irregular manifold of type $2_{Bd}$ for $(T_{j})_{j\in {\mathbb N}},$
and particularly, $(T_{j})_{j\in {\mathbb N}}$
is densely $m_{n}$-distributionally chaotic of type $2_{Bd}.$
\end{thm}

\begin{proof}
Without loss of generality, we may assume that \eqref{grozno} holds and $m=1.$ Since $\overline{Bd}_{l:m_{n}}(B)=c$, we can find two strictly increasing sequences $(s_{k})_{k\in {\mathbb N}}$ and $(j_{k})_{k\in {\mathbb N}}$ of positive integers such that for each $k\in {\mathbb N}$ we have that the set $B_{k}:=B \cap [j_{k}+1,j_{k}+m_{s_{k}}]$ contains at least $s_{k}(c-k^{-1})$ integers. Set $B':=\bigcup_{k\in {\mathbb N}}B_{k}.$
Then 
$$
\frac{1}{s_{k}}\sup_{n\in {\mathbb N}}\Bigl| B' \cap [n+1,n+m_{s_{k}}]\Bigr|\geq c\bigl(1-k^{-1} \bigr),\quad k\in {\mathbb N},
$$
which implies 
$$
\overline{Bd}_{l:m_{n}}(B')=\liminf_{s\rightarrow \infty}\frac{\sup_{n\in {\mathbb N}} \bigl| B' \cap [n+1,n+m_{s}]\bigr|}{s}\geq c.
$$
Now we can construct a sequence $(x_{k})_{k\in {\mathbb N}}$ in $X_{0}$
such that, for every $k\in {\mathbb N},$ one has: $p_{k}(x_{k})\leq 1,$ $p_{Y}^{1}(T_{j}x_{k})\geq k2^{k}$ for $j\in B_{k},$ $p_{Y}^{k}(T_{j}x_{s})\leq 1/(k+1)$ for all $s\in {\mathbb N}_{k-1}
$ and $j\in {\mathbb N}$ with 
$j\in B_{k},$ as well as: 
$$
\frac{\bigl| \{1\leq j \leq m_{j_{k}} : p_{Y}^{k}(T_{j}x_{s})<1/k \}\bigr|}{j_{k}}\geq 1-k^{-2},\quad s\in {\mathbb N}_{k-1}.
$$
The final conclusion follows similarly as in the proof of Theorem \ref{na-dobro}.
\end{proof}

We can simply formulate corollaries of Theorem \ref{na-dobrorc}, Theorem \ref{na-dobro2} and Theorem \ref{na-dobro3} for dense $\lambda$-reiterative distributional chaos of type $1^+,$ dense $\lambda$-distributional chaos of type $2$ and dense $\lambda$-distributional chaos of type $2_{Bd}, $ respectively, as well as corollaries of  Theorem \ref{na-dobrorc}, Theorem \ref{na-dobro1}, Theorem \ref{na-dobro2} and Theorem \ref{na-dobro3} for dense ($\lambda-$) $m_{n}$-reiterative distributional chaos of type $1^+,$ certain types of dense Li-Yorke chaos, dense  ($\lambda-$) $m_{n}$-distributional chaos of type $2$ and dense ($\lambda-$) $m_{n}$-distributional chaos of type $2_{Bd}, $ respectively, for linear unbounded operators (see Corollary \ref{na-dobrorade} and Corollary \ref{cea1}).
Concerning the possible applications of Theorem \ref{na-dobro}, Theorem \ref{na-dobrorc}, Theorem \ref{na-dobro1}, Theorem \ref{na-dobro2} and Theorem \ref{na-dobro3} and these corollaries to unbounded linear operators, we would like to note that the additional use of regularizing operator $C\in L(X)$ does not take a right effect in the investigation of dense Li-Yorke chaos, in contrast with the notion of dense distributional chaos (cf. the second part of \cite[Corollary 3.12]{mendoza}). On the other hand, there exists
a great number of possible applications of the above-mentioned results to the sequences of bounded linear operators in Banach spaces, and we will present here only one illustrative:

\begin{example}\label{rdc}
Let $X:=l^{1},$ $B\subseteq {\mathbb N}$ and $\overline{Bd}(B)=1.$ Further on, let for each $j\in {\mathbb N}$ we have that $(\omega_{n}^{j})_{n\in {\mathbb N}}$ is a bounded sequence of positive reals such that $\omega_{1}^{j}>j^{3}$ for all $j\in B,$ and let for each $j\in {\mathbb N}$ we have that $(a_{n}^{j})_{n\in {\mathbb N}}$ is a strictly increasing sequence of positive integers such that 
the sequence $(a_{1}^{j})_{j\in {\mathbb N}}$ is strictly increasing, as well.
Set
$$
T_{j}\bigl \langle x_{n} \bigr \rangle_{n\in {\mathbb N}}:=\bigl \langle \omega_{1}^{j} x_{a_{1}^{j}}, \omega_{2}^{j} x_{a_{2}^{j}},\cdot \cdot \cdot \bigr \rangle,\quad \bigl \langle x_{n} \bigr \rangle_{n\in {\mathbb N}} \in X.
$$  
Define $x_{n}:=0$ if $n\notin \bigcup_{j\in B}a_{1}^{j}$ and $x_{n}:=j^{-2}$ if $n=a_{1}^{j}$ for some $j\in B.$ Then it can be easily seen that the vector $\langle x_{n}  \rangle_{n\in {\mathbb N}}$ is reiteratively distributionally unbounded of type $1^+$ for $(T_{j})_{j\in {\mathbb N}}.$ By Theorem \ref{na-dobro}, it readily follows that the sequence $(T_{j})_{j\in {\mathbb N}}$ is reiteratively distributionally chaotic of type $1^+.$
\end{example}

\subsection{An application to abstract partial differential equations}\label{PDEs}

It is almost straightforward and rather technical to transfer all results proved in this section by now for operator families defined on the non-negative real axis. For the sake of brevity and better exposition, we will consider here only continuous analogues of Theorem \ref{na-dobro} and Corollary \ref{na-dobrorade}-Corollary \ref{na-dobrwo}.

Let $T(t) : D(T(t)) \subseteq X \rightarrow Y$ be a linear possibly not continuous mapping ($t\geq 0$). By $Z(T)$ we denote the set of all $x\in X$ such that
$x\in D(T(t))$ for all $t\geq 0$ as well as that the mapping $t\mapsto T(t)x,$ $t\geq 0$ is continuous.  
Denote by $m(\cdot)$ the Lebesgue measure on $[0,\infty)$ and by ${\mathrm F}$ the class consisting of all  increasing mappings $f : [0,\infty) \rightarrow [1,\infty)$ satisfying that  $\liminf_{t\rightarrow +\infty}\frac{f(t)}{t}>0.$

We will use the following continuous counterpart of Definition \ref{prckojed}:

\begin{defn}\label{prckojed-prim} (\cite{F-operatori})
Let $A\subseteq [0,\infty)$, let $f \in {\mathrm F},$ and let $q\in [1,\infty).$ Then: 
\begin{itemize}
\item[(i)] The lower $f$-density of $A,$ denoted by $\underline{d}_{f}(A),$ is defined through:
$$
\underline{d}_{f} (A):=\liminf_{t\rightarrow \infty}\frac{m(A \cap [0,f(t)])}{t}.
$$
\item[(ii)] The lower $qc$-density of $A,$ denoted by $\underline{d}_{qc}(A),$ is defined through:
$$
\underline{d}_{qc}(A):=\liminf_{t\rightarrow \infty}\frac{m(A \cap [0,t^{q}])}{t}.
$$

\end{itemize}
\end{defn}

We introduce the notion of $\tilde{X}_{f}$-distributional chaos as follows:

\begin{defn}\label{DC-unbounded-fric-cont} 
Suppose that $\tilde{X}$ is a non-empty subset of $X,$ $T(t) : D(T(t)) \subseteq X \rightarrow Y$ is a linear possibly not continuous mapping ($t\geq 0$) and $f \in {\mathrm F}.$ If there exist an uncountable
set $S\subseteq Z(T) \cap \tilde{X}$ and
$\sigma>0$ such that for each $\epsilon>0$ and for each pair $x,\
y\in S$ of distinct points we have that 
\begin{align*}
\begin{split}
& \underline{d}_{f}\Bigl( \bigl\{t\geq 0 :
d_{Y}\bigl(T(t)x,T(t)y\bigr)< \sigma \bigr\}\Bigr)=0,
\\
& \underline{d}_{f}\Bigl(\bigl\{t\geq 0 : d_{Y}\bigl(T(t)x,T(t)y\bigr)
\geq \epsilon \bigr\}\Bigr)=0,
\end{split}
\end{align*}
then we say that $(T(t))_{t\geq 0}$ is
$\tilde{X}_{f}$ distributionally chaotic ($f$-distributionally chaotic, if $\tilde{X}=X$). 
Furthermore, we say that  $(T(t))_{t\geq 0}$ is densely
$\tilde{X}_{f}$-distributionally chaotic iff $S$ can be chosen to be dense in $\tilde{X}.$
The set $S$ is said to be ${\sigma_{\tilde{X}}}_{f}$-scrambled set ($\sigma_{f}$-scrambled set in the case that $\tilde{X}=X$)     
of $(T(t))_{t\geq 0}.$

If $q\geq 1$ and $f(t):=1+t^{q}$ ($t\geq 0$), then we particularly obtain the notions of (dense) $\tilde{X}_{q}$-distributional chaos, (dense) $q$-distributional chaos, ${\sigma_{\tilde{X}}}_{q}$-scrambled set and $\sigma_{q}$-scrambled set for $(T(t))_{t\geq 0}.$
\end{defn}

The basic result for applications is the following counterpart of Theorem \ref{na-dobro}, whose proof is very similar to that of afore-mentioned theorem and therefore omitted (cf. also the proof of \cite[Theorem 4.1]{mendoza}):

\begin{thm}\label{na-dobro-cont}
Suppose that $X$ is separable, $f \in {\mathrm F}$ and $m_{n}:=\lceil f(n) \rceil,$ $n\in {\mathbb N}.$ Suppose, further, that $(T(t))_{t\geq 0}\subseteq L(X,Y)$ is strongly continuous,
$X_{0}$ is a dense linear subspace of $X,$ as well as:
\begin{itemize}
\item[(i)] $\lim_{t\rightarrow \infty}T(t)x=0,$ $x\in X_{0},$
\item[(ii)] there exist a vector $y\in X,$ a set $B\subseteq [0,\infty)$ and a number $m\in {\mathbb N}$ such that
\begin{align*}
\liminf_{t\rightarrow \infty}\frac{f(t)-\bigl| B \cap [1,f(t)]\bigr|}{t}=0
\end{align*}
and
\begin{align}\label{kamenja-cont}
\lim_{t\rightarrow \infty,t\in B}p^{m}_{Y}\bigl( T(t)y \bigr)=+\infty.
\end{align} 
\end{itemize}
Then $(T(t))_{t\geq 0}$ is densely $f$-distributionally chaotic and the corresponding scrambled set $S$ can be chosen to be a dense uniformly $f$-distributionally irregular submanifold of $X.$ Furthermore, for each fixed number $t_{0}>0,$ we have that the operator $T(t_{0})$ 
is densely $m_{n}$-distributionally chaotic, and moreover, the corresponding scrambled set $S_{t_{0}}$ can be chosen to be a dense uniformly $m_{n}$-distributionally irregular submanifold of $X.$
\end{thm} 

We also attach the following obvious counterparts of Corollary \ref{na-dobrorade}-Corollary \ref{na-dobrwo}:

\begin{cor}\label{na-dobrorade-cont}
Suppose that $X$ is separable, $\lambda \in (0,1],$ $(T(t))_{t\geq 0}\subseteq L(X,Y)$ is strongly continuous,
$X_{0}$ is a dense linear subspace of $X,$ as well as:
\begin{itemize}
\item[(i)] $\lim_{t\rightarrow \infty}T(t)x=0,$ $x\in X_{0},$
\item[(ii)] there exist a vector $y\in X,$ a set $B\subseteq  [0,\infty)$ and a number $m\in {\mathbb N}$ such that
\begin{align*}
\liminf_{n\rightarrow \infty}\frac{t^{1/\lambda}-\bigl| B \cap [1,t^{1/\lambda}]\bigr|}{t}=0
\end{align*}
and \eqref{kamenja-cont} holds. 
\end{itemize}
Then $(T(t))_{t\geq 0}$ is densely $\lambda$-distributionally chaotic and the corresponding scrambled set $S$ can be chosen to be a dense uniformly $\lambda$-distributionally irregular submanifold of $X.$ Furthermore, for each fixed number $t_{0}>0,$ we have that the operator $T(t_{0})$ 
is densely $m_{n}$-distributionally chaotic, and moreover, the corresponding scrambled set $S_{t_{0}}$ can be chosen to be a dense uniformly $m_{n}$-distributionally irregular submanifold of $X;$ here, $m_{n}:=\lceil n^{q} \rceil$ for all $n\in {\mathbb N}.$
\end{cor}

\begin{cor}\label{na-dobrwo-cont}
Suppose that $X$ is separable, $(T(t))_{t\geq 0}\subseteq L(X,Y)$ is strongly continuous, $f \in {\mathrm F},$
$X_{0}$ is a dense linear subspace of $X,$ as well as:
\begin{itemize}
\item[(i)]  $\lim_{t\rightarrow \infty}T(t)x=0,$ $x\in X_{0},$
\item[(ii)] there exist a vector $y\in X$ and a number $m\in {\mathbb N}$ such that \eqref{kamenja-cont} holds with $B={\mathbb N}.$
\end{itemize}
Then $(T(t))_{t\geq 0}$ is densely $f$-distributionally chaotic and the corresponding scrambled set $S$ can be chosen to be a dense uniformly $f$-distributionally irregular submanifold of $X.$ Furthermore, for each fixed number $t_{0}>0,$ we have that the operator $T(t_{0})$ 
is densely $m_{n}$-distributionally chaotic, and moreover, the corresponding scrambled set $S_{t_{0}}$ can be chosen to be a dense uniformly $m_{n}$-distributionally irregular submanifold of $X;$ here, $m_{n}:=\lceil f(n) \rceil$ for all $n\in {\mathbb N}.$
\end{cor}

In the next section, we shall consider possible applications of results established in this subsection to the abstract first order differential equations.
We continue with the observation that
we have not used any semigroup property of operator family $(T(t))_{t\geq 0}$
under our consideration so that
the results of this subsection are applicable in the qualitative analysis of solutions for certain classes of the abstract (multi-term) fractional differential equations with Caputo derivatives. Here we will present only one example of possible application
of this type; for more details how we can incorporate Theorem \ref{na-dobro-cont} and Corollary \ref{na-dobrorade-cont}-Corollary \ref{na-dobrwo-cont} in the qualitative analysis of solutions to abstract fractional PDEs, the reader may consult
\cite[Section 3.3]{knjigaho}:

\begin{example}\label{frakcione}
(cf. \cite{ji} and \cite{NSJOM} for the notion) Suppose that $X$ is a symmetric space of non-compact type and rank one,
$p>2,$ the parabolic domain $P_{p}$ and the positive real number
$c_{p}$ possess the same meaning as in \cite{ji}. Suppose, further, that $\Delta_{X,p}^{\natural}$ denotes the corresponding
Laplace-Beltrami operator and 
$P(z)=\sum_{j=0}^{n}a_{j}z^{j},$ $z\in {\mathbb C}$ is a
non-constant complex polynomial with $a_{n}>0.$ Consider the abstract fractional Cauchy problem:
\begin{align*}
\begin{split}
& {\mathbf D}_{t}^{2a}u(t)+ cu(t)= -e^{i\theta}P(\Delta_{X,p}^{\natural}){\mathbf
D}_{t}^{a}u(t),
\quad t \geq 0, \\
& u^{(k)}(0)=u_k,\quad  k=0,\cdot \cdot \cdot, \lceil 2a
\rceil -1,
\end{split}
\end{align*}
where $0<a<2,$ $c>0$ and
$|\theta|<\min(\frac{\pi}{2}-n\arctan
\frac{|p-2|}{2\sqrt{p-1}},\frac{\pi}{2}-n\arctan
\frac{|p-2|}{2\sqrt{p-1}}-\frac{\pi}{2}a).$ Then $-e^{i\theta}P(\Delta_{X,p}^{\natural})$
generates an exponentially bounded, analytic
resolvent propagation family $((R_{\theta,P,0}(t))_{t\geq 0},\cdot
\cdot \cdot, (R_{\theta,P,\lceil 2a \rceil -1}(t))_{t\geq 0})$ of certain
angle; see \cite{knjigaho} for the notion. Applying Corollary \ref{na-dobrwo-cont} and the analysis from \cite[Example 2.8]{NSJOM}, we can show that the condition
$$
-e^{i\theta}P\bigl(\mbox{int}\bigl(P_{p}\bigr)\bigr) \ \cap \
\Bigl\{\bigl(it\bigr)^{a}+c\bigl(it\bigr)^{-a} : t\in {\mathbb
R} \setminus \{0\}\Bigr\}\ \neq \emptyset
$$
implies that $(R_{\theta,P,0}(t))_{t\geq 0}$ is densely $f$-distributionally chaotic and for each $t_{0}>0$ the operator $T(t_{0})$ 
is densely $m_{n}$-distributionally chaotic, where $m_{n}:=\lceil f(n) \rceil$ for all $n\in {\mathbb N}$ ($f \in {\mathrm F}$).
\end{example}

\section{Conclusions, final remarks and open problems}\label{problemen}

In this section, we provide several observations and remarks about results obtained so far and ask some questions. 
In the considerations of 
backward shift operators, we will always assume that $X$ is  a Fr\' echet sequence space in which $(e_{n})_{n\in {\mathbb N}}$ is basis and $(\omega_{n})_{n\in {\mathbb N}}$ is a sequence of positive weights; furthermore, we will always assume that
the unilateral weighted backward shift $T_{\omega},$
given by \eqref{oladilo},  
is a continuous linear operator on $X.$ Recall that the finite linear combinations of vectors 
from the basic $(e_{n})_{n\in {\mathbb N}}$ form a dense submanifold of $X.$

First of all, we would like to ask the following questiones:

\begin{prob}\label{nijelako}
Suppose $(m_{n}) \in {\mathrm R}$ and $T_{\omega}$ is distributionally chaotic. Is it true that 
$T_{\omega}$ is $m_{n}$-distributionally chaotic?
\end{prob}

\begin{prob}\label{malolie}
Let $(m_{n}) \in {\mathrm R}$, let $T\in L(X)$ satisfy that there exists a dense submanifold $X_{0}$ of $X$ such that $\lim_{n\rightarrow \infty}T^{n}x=0$ for all $x\in X_{0},$ and let $T$ be distributionally chaotic. Is it true that 
$T$ is $m_{n}$-distributionally chaotic?
\end{prob}

Concerning these problems, irrelevant of the fact whether the answers to them  are affirmative or not, we would like to note that combining Theorem \ref{ruza} and Theorem \ref{na-dobro} immediately yields the following extension of  \cite[Theorem 25]{2013JFA}:

\begin{thm}\label{djubre}
Suppose that $(m_{n}) \in {\mathrm R}$ and $T\in L(X)$ satisfies that there exists a dense submanifold $X_{0}$ of $X$ such that $\lim_{n\rightarrow \infty}T^{n}x=0$ for all $x\in X_{0}.$ Then the following assertions are equivalent:
\begin{itemize}
\item[(i)] $T$ is $m_{n}$-distributionally chaotic. 
\item[(ii)] $T$ is densely $m_{n}$-distributionally chaotic (of type $1$). 
\item[(iii)] There exists an $m_{n}$-distributionally unbounded (irregular) vector for $T.$
\item[(iv)] There exists a dense uniformly $m_{n}$-distributionally irregular submanifold for $T.$
\end{itemize}
\end{thm}

The following generalization of \cite[Theorem 26, Corollary 27]{2013JFA} can be proved as for distributional chaos (similarly we can reconsider the statements \cite[Theorem 29-Theorem 30, Corollary 31-Corollary 32]{2013JFA} for $m_{n}$-distributional chaos
in Fr\' echet sequence spaces in which $(e_{n})_{n\in {\mathbb Z}}$ is a basis; with the exception of \cite[Problem 23]{2013JFA}, we obtain further extensions of all other statements established in \cite[Section 3]{2013JFA} for sequences of operators):

\begin{thm}\label{da-se-ohladi-rsd}
\begin{itemize}
\item[(i)] Suppose that $(m_{n}) \in {\mathrm R}$ and the operator $T$ is given by \eqref{oladilo} with the weight $w_{n}\equiv 1$ ($n\in {\mathbb N}$). Let there exist a subset $S$ of natural numbers such that
the series $\sum_{n\in S}e_{n}$ converges in $X,$ and $\underline{d}_{m_{n}}(S^{c})=0.$ Then the operator $T$ is densely $m_{n}$-distributionally
chaotic. 
\item[(ii)] Suppose that $(m_{n}) \in {\mathrm R}$ and the operator $T_{\omega}$ satisfies that there exists a subset  $S$ of natural numbers such that
the series $\sum_{n\in S}(\prod_{i=1}^{n}\omega_{i})^{-1}e_{n}$ converges in $X,$ and $\underline{d}_{m_{n}}(S^{c})=0.$ Then the operator $T$ is densely $m_{n}$-distributionally
chaotic. 
\end{itemize}
\end{thm}

\begin{cor}\label{rsd-denari}
\begin{itemize}
\item[(i)] Suppose that $\lambda \in (0,1]$ and the operator $T$ is given by \eqref{oladilo} with the weight $w_{n}\equiv 1$ ($n\in {\mathbb N}$). Let there exist a subset $S$ of natural numbers such that
the series $\sum_{n\in S}e_{n}$ converges in $X,$ and $\underline{d}_{1/\lambda}(S^{c})=0.$ Then the operator $T$ is densely $\lambda$-distributionally
chaotic. 
\item[(ii)] Suppose that $\lambda \in (0,1]$ and the operator $T_{\omega}$ satisfies that there exists a subset  $S$ of natural numbers such that
the series $\sum_{n\in S}(\prod_{i=1}^{n}\omega_{i})^{-1}e_{n}$ converges in $X,$ and $\underline{d}_{1/\lambda}(S^{c})=0.$ Then the operator $T$ is densely $\lambda$-distributionally
chaotic. 
\end{itemize}
\end{cor}

Now we would like to ask the following:

\begin{prob}\label{rucak}
Suppose that $X:=l^{p}({\mathbb N})$ for some $p\in [1,\infty)$ or $X:=c_{0}({\mathbb N}),$ and $\lambda \in (0,1].$ Is it true that there exists a densely $\lambda$-distributionally chaotic backward shift operator $T_{\omega}$ which is $\lambda$-distributionally chaotic and not $\lambda'$-distributionally chaotic for any number $\lambda '\in (0,\lambda)?$ 
\end{prob}

Regarding distributionally chaotic backward shift operators, mention should be also made of papers \cite{gimenez-p}-\cite{gimenez}, \cite{afa-raj} and \cite{countable}. For the sake of brevity, we will not reconsider the related problematic for $m_{n}$-distributional chaos here.

Concerning applications to the abstract partial differential equations of first order whose solutions are governed by strongly continuous semigroups, it is clear that our results from Subsection \ref{PDEs} can be employed at any place where the Desch-Schappacher-Webb criterion \cite{fund} is employed (see e.g. \cite{barahina1}, \cite{mendoza} and references cited therein); concerning applications to the abstract ill-posed partial differential equations of first order whose solutions are governed by fractionally integrated $C$-semigroups, our results can be used to the equations considered in \cite{mendoza} and \cite[Subsection 3.1.4]{knjigah}. 
But, if we are in a position, for example, in which the requirements of the Desch-Schappacher-Webb criterion holds for a strongly continuous semigroup $(T(t))_{t\geq 0},$ then we always have the existence of a dense linear subspace $X_{0}$ of $X$ satisfying $\lim_{t\rightarrow \infty}T(t)x=0$ for all $x\in X_{0},$ so that it is quite natural to ask the following:

\begin{prob}\label{DSWprojlim}    
Suppose that $(T(t))_{t\geq 0}$ is a distributionally chaotic, strongly continuous semigroup on $X$ and there exists a dense linear subspace $X_{0}$ of $X$ satisfying $\lim_{t\rightarrow \infty}T(t)x=0$ for all $x\in X_{0}.$ Is it true that 
$(T(t))_{t\geq 0}$ is $f$-distributionally chaotic for all $f\in {\mathrm F}?$
\end{prob}

It seems very plausible that \cite[Theorem 4.2]{mendoza} admits a reformulation for $m_{n}$-distributional chaos, so that a positive solution to Problem \ref{malolie} immediately answers Problem \ref{DSWprojlim} in the affirmative (observe, however, that it is not clear how one can reconsider the above-mentioned theorem for fractional solution operator families).

Suppose finally that $1\leq p<\infty.$ 
We refer the reader to 
\cite[Definition 4.3]{fund} for the notions of an admissible weight function $\rho : [0,\infty) \rightarrow (0,\infty)$ and
the Banach spaces $L_{\rho}^{p}([0,\infty), {\mathbb K}),$ $C_{0,\rho}([0,\infty), {\mathbb K}).$
We close the paper by 
proposing the following continuous counterpart of Problem \ref{rucak}:

\begin{prob}\label{rucak-cont}
Suppose that $\lambda \in (0,1].$ Can we find
an admissible weight function $\rho(\cdot)$ 
and a strongly continuous semigroup $(T(t))_{t\geq 0}$ on $X:=L_{\rho}^{p}([0,\infty), {\mathbb K})$ or $X:=C_{0,\rho}([0,\infty), {\mathbb K}),$ which is $\lambda$-distributionally chaotic and not $\lambda'$-distributionally chaotic for any number $\lambda '\in (0,\lambda)?$
\end{prob}


\begin{thebibliography}{90}

\bibitem{barahina1}
\textsc{X. Barrachina, J. A. Conejero,}
\emph{Devaney chaos and distributional chaos in the solution of certain partial
differential equations,}
Abstr. Appl. Anal. 2012, Art. ID 457019, 11 pp.

\bibitem{bayart}
\textsc{F. Bayart, E. Matheron,}
\emph{Dynamics of Linear Operators,}
Cambridge Tracts in Mathematics, Cambridge University
Press, Cambridge, UK, \textbf{179(1)}, 2009.

\bibitem{BayaImre}
\textsc{F. Bayart, I. Z. Ruzsa,}
\emph{Difference sets and frequently hypercyclic weighted shifts,}
Ergodic Theory Dynam. Systems \textbf{35} (2015), 691--709.

\bibitem{bea-b}
\textsc{B. Beauzamy,}
\emph{Introduction to Operator Theory and Invariant Subspaces,} North-Holland, Amsterdam,
1988.

\bibitem{2011}
\textsc{T. Berm\'udez, A. Bonilla, F. Martinez-Gimenez, A. Peris,}
\emph{Li-Yorke and distributionally chaotic operators,}
J. Math. Anal. Appl. \textbf{373} (2011),  83--93.

\bibitem{2013JFA}
\textsc{N. C. Bernardes Jr., A. Bonilla, V. M\"uler, A. Peris,}
\emph{Distributional chaos for linear operators,} J. Funct. Anal.
\textbf{265} (2013),  no. 1, 2143--2163.

\bibitem{2018JMMA}
\textsc{N. C. Bernardes Jr., A. Bonilla, A. Peris, X. Wu,}
\emph{Distributional chaos for operators on Banach spaces,}
J. Math. Anal. Appl. {\bf 459} (2018), 797--821.

\bibitem{band}
\textsc{N. C. Bernardes Jr, A. Bonilla, V. M\"uller, A. Peris,}
\emph{Li-Yorke chaos in linear dynamics,}
Ergodic Theory Dynam. Systems
35 (2015), 1723--1745.

\bibitem{ARXIV}
\textsc{N. C. Bernardes Jr., A. Bonilla, A. Peris,}
\emph{Mean Li-Yorke chaos in Banach spaces,}
preprint, arXiv:1804.03900.

\bibitem{bk}
\textsc{A. Bonilla, M. Kosti\' c,}
\emph{Reiterative distributional chaos on Banach spaces,}
preprint, arXiv:1903.09143.

\bibitem{mendoza}
\textsc{J. A. Conejero, M. Kosti\' c, P. J. Miana, M. Murillo-Arcila,}
\emph{Distributionally chaotic families of operators on Fr\' echet spaces,}
Comm. Pure Appl. Anal. {\bf 15} (2016), 1915--1939.

\bibitem{fund}
\textsc{W. Desch, W. Schappacher, G. F. Webb}, \emph{Hypercyclic
and chaotic semigroups of linear operators}, Ergodic Theory Dynam. Systems
\textbf{17} (1997), 1--27.

\bibitem{down}
\textsc{T. Downarowicz,}
\emph{Positive topological entropy implies chaos \emph{DC2}},
Proc. Amer. Math. Soc. {\bf 142} (2013), 137--149.

\bibitem{gimenez-p}
\textsc{F. Mart\'inez-Gim\'enez, P. Oprocha, A. Peris,}
\emph{Distributional chaos for backward shifts,} 
J. Math. Anal. Appl.
{\bf 351} (2009), 607--615.

\bibitem{gimenez}
\textsc{F. Mart\'inez-Gim\'enez, P. Oprocha, A. Peris,}
\emph{Distributional chaos for operators with full scrambled sets,}
Math. Z. \textbf{274} (2013), 603--612.

\bibitem{godefroy}
\textsc{J. Godefroy, J. H. Shapiro,}
\emph{Operators with dense, invariant,
cyclic vector manifolds,} 
J. Funct. Anal. {\bf 98} (1991), 229--269.

\bibitem{GMM}
\textsc{S. Grivaux, \'E. Matheron, Q. Menet},
\emph{Linear dynamical systems on Hilbert spaces: typical properties and explicit examples},
Memoirs of Amer. Math. Soc, in press.

\bibitem{erdper}
\textsc{K.-G. Grosse-Erdmann, A. Peris,}
\textit{Linear Chaos},
Springer-Verlag, London, 2011.

\bibitem{ji}
\textsc{L. Ji, A. Weber,}
\emph{Dynamics of the heat semigroup on
symmetric spaces,}
Ergodic Theory Dynam. Systems {\bf 30} (2010),
457--468.

\bibitem{knjigah}
\textsc{M. Kosti\'c,}
\emph{Generalized Semigroups and Cosine Functions},
Mathematical Institute SANU, Belgrade, 2011.

\bibitem{knjigaho}
\textsc{M. Kosti\'c,} {\it Abstract Volterra Integro-Differential
Equations,} CRC Press, Boca Raton, Fl., 2015.

\bibitem{novascience}
\textsc{M. Kosti\'c,} {\it Chaos for Linear Operators and Abstract Differential
Equations,} Nova Science Publishers Inc., accepted.

\bibitem{NSJOM}
\textsc{M. Kosti\'c,} 
{\it Distributionally chaotic properties of abstract fractional differential equations,}
Novi Sad J. Math. {\bf 45} (2015), 201--213.

\bibitem{nis-dragan}
\textsc{ M. Kosti\' c,}
\emph{Li-Yorke chaotic properties of abstract differential equations of first order,}
Appl. Math. Comput. Sci. {\bf 1} (2016), 15--26.

\bibitem{F-operatori}
\textsc{M. Kosti\'c,} 
{\it ${\mathcal F}$-Hypercyclic operators on Fr\' echet spaces,}
preprint, arXiv:1809.02549.

\bibitem{ddc}
\textsc{M. Kosti\' c,}
\emph{Disjoint distributional chaos in Fr\'echet spaces,}
preprint, arXiv:1812.03824.

\bibitem{ddc-emen-enes}
\textsc{M. Kosti\' c,}
\emph{Disjoint reiterative $m_{n}$-distributional chaos in Fr\'echet spaces,}
preprint, https://www.researchgate.net/publication/332752441.

\bibitem{ddc-ly}
\textsc{M. Kosti\' c,}
\emph{Disjoint Li-Yorke chaos in Fr\'echet spaces,}
preprint, arXiv: 1905.03992.

\bibitem{skopje}
\textsc{M. Kosti\' c, D. Velinov,}
\emph{Reiterative $(m_{n})$-distributional chaos for binary relations over metric spaces,}
Mat. Bilten, in press.

\bibitem{turkish-notes}
\textsc{L. Luo, B. Hou,}
\emph{Some remarks on distributional chaos for bounded linear operators,}
Turkish J. Math. {\bf 39} (2015), 251--258.

\bibitem{menet}
\textsc{Q.  Menet,}
\emph{Linear  chaos  and  frequent  hypercyclicity,} Trans. Amer. Math. Soc. \textbf{369} (2017), 4977--4994.

\bibitem{wuja}
\textsc{X. Wu,}
\emph{Maximal distributional chaos of weighted shift operators on K\"othe sequence spaces,}
Czech. Math. J. {\bf 64} (2014), 105--114.

\bibitem{xwuchen}
\textsc{X. Wu, G. Chen, P. Zhu,}
\emph{Invariance of chaos from backward shift on the K\"othe sequence space,}
Nonlinearity {\bf 27} (2014), 271. https://doi.org/10.1088/0951-7715/27/2/271.

\bibitem{afa-raj}
\textsc{X. Wu, L. Wang, G. Chen,}
\emph{Weighted backward shift operators with invariant distributionally scrambled
subsets,} Ann. Fuct. Anal. {\bf 8} (2017), 199--210. 

\bibitem{xwu}
\textsc{X. Wu, P. Zhu,}
\emph{Li-Yorke chaos of backward shift operators
on K\"othe sequence spaces,}
Topology Its Appl. {\bf 160} (2013), 924--929.

\bibitem{xiong}
\textsc{J. C. Xiong, H. M. Fu, H. Y. Wang,}
\emph{A class of Furstenberg families and their applications to chaotic
dynamics,} Sci. China Math. {\bf 57} (2014), 823--836.

\bibitem{countable}
\textsc{Z. Yin, S. He, Y. Huang,}
\emph{On Li-Yorke and distributionally chaotic direct sum operators,}
Topology Appl. {\bf 239} (2018), 35--45. 


\end{thebibliography}
\end{document}